\documentclass{amsart}

\usepackage{amsmath,amssymb,amsfonts,amsthm}
\usepackage{enumerate}
\allowdisplaybreaks

\newtheorem{theorem}{Theorem}[section]
\newtheorem{lemma}[theorem]{Lemma}
\newtheorem{conj}[theorem]{Conjecture}
\newtheorem{proposition}[theorem]{Proposition}
\newtheorem{corollary}[theorem]{Corollary}
\newtheorem{example}[theorem]{Example}
\newtheorem{question}[theorem]{Question}

\numberwithin{equation}{section}

\theoremstyle{remark}
\newtheorem{remark}[theorem]{Remark}

\newcommand{\Ric}{\mathop{\mathrm{Ric}}\nolimits}
\newcommand{\Ad}{\mathop{\mathrm{Ad}}\nolimits}

\newcommand{\ecal}{\mathcal{E}}

\newcommand{\mcal}{\mathcal{M}}

\def\calM{\mcal}  
  
  \def\calE{\ecal}

\def\a{\alpha} 
 
\def\th{\theta} 
 \def\eps{\epsilon}

\def\max{{\operatorname{max}}}

\def\ad{\hbox{\rm ad}}
\def\Ad{\hbox{\rm Ad}}

\def\K{K\"ahler } 
\def\KE{K\"ahler--Einstein }

\def\Ric{\hbox{\rm Ric}\,}

\def\h#1{\hbox{#1}}

\def\strutdepth{\dp\strutbox}
\def\specialstar{\vtop to \strutdepth{
    \baselineskip\strutdepth
    \vss\llap{$\star$\ \ \ \ \ \ \ \ \  }\null}}
\def\marginalstar{\strut\vadjust{\kern-\strutdepth\specialstar}}

\def\marginal#1{\strut\vadjust{\kern-\strutdepth
    {\vtop to \strutdepth{
    \baselineskip\strutdepth
    \vss\llap{{ \small #1 }}\null}
    }}
    }

\def\text{\textstyle}

\def\q{\quad} \def\qq{\qquad}

\def\ra{\rightarrow}

 \newcommand{\NN}{{\mathbb N}}

\def\beq{\begin{equation}}
\def\eeq{\end{equation}}
\def\bpf{\begin{proof}}
\def\epf{\end{proof}}

\def\eaeq{\end{aligned}}
\def\baeq{\begin{aligned}}

\def\mf{\mathfrak}

\author{Artem Pulemotov and Yanir A. Rubinstein}

\title{Ricci iteration on homogeneous spaces}

\begin{document}

\maketitle

\begin{center}
{\it To the memory of Itzhak Bar-Lewaw Mulstock and Serhiy Ovsiyenko}
\end{center}

\begin{abstract}
The Ricci iteration is a discrete analogue of the Ricci flow. We give the first study of the Ricci iteration on a class of Riemannian manifolds that are not K\"ahler. The Ricci iteration in the non-K\"ahler setting exhibits new phenomena. Among them is the 
existence of so-called {\it ancient Ricci iterations}. As we show, these are closely related to ancient Ricci flows and provide the first nontrivial examples of Riemannian metrics 
to which the Ricci operator can be applied infinitely many times. 
In some of the cases we study, these ancient Ricci iterations emerge 
(in the Gromov--Hausdorff topology) from a collapsed Einstein metric and converge smoothly to a second Einstein metric.
In the case of compact homogeneous spaces with maximal isotropy, we prove
a relative compactness result that excludes collapsing.
Our work can also be viewed as proposing 
a dynamical criterion for detecting whether an ancient Ricci flow exists on a given Riemannian manifold as well as a method for predicting 
its limit.
\end{abstract}

\def\lb{\label}

\section{Introduction}

Let $(M,g_1)$ be a smooth Riemannian manifold. 
A Ricci iteration is a sequence of metrics $\{g_i\}_{i\in\mathbb N}$ on $M$ satisfying 
\begin{align}
\label{RIEq}
\Ric g_{i+1}=g_i,\q i\in\mathbb N,
\end{align}
where $\Ric g_{i+1}$ denotes the Ricci curvature of $g_{i+1}$. 
One may think of~\eqref{RIEq}
as a dynamical system on the space of Riemannian  metrics on $M$. We restrict our attention to the case of positive Ricci curvature in the present article; different Ricci iterations can be defined in the context of non-positive curvature.

Part of the interest in the Ricci iteration is that, clearly, 
Einstein metrics with Einstein constant~1 are fixed points, and so~\eqref{RIEq} aims to provide a natural new approach to uniformization. 
In essence, the Ricci iteration aims to reduce the Einstein equation to a sequence of prescribed Ricci curvature equations.
Introduced by the second-named author \cite{R07,R08} as a  discretization of the Ricci flow, the Ricci iteration 
has been since studied by a number of 
authors, see the survey~\cite[\S6.5]{R14} and references therein.
In all previous works, the underlying manifold $(M,g_1)$ is assumed to be K\"ahler;
essentially nothing is known about the Ricci iteration in the general Riemannian setting.

Given the central r\^ole of uniformization and of the Ricci flow 
in geometry, it seems of interest to investigate whether the Ricci iteration could be understood for general Riemannian manifolds which may not be K\"ahler. 
In this article we take a first 
step in this direction. Namely, we show that, indeed, the Ricci iteration can be defined on some non-K\"ahler manifolds, and that under some natural assumptions it converges to an Einstein metric with positive Ricci curvature. We also prove a relative compactness result. The manifolds we investigate here are compact homogeneous spaces. The study of Einstein metrics and the Ricci flow on such spaces is an active field,  see, e.g.,~\cite[Chapter~7]{AB87}, \cite{JIMJ92,XCLSC09,DGTP10,CBMWWZ04,CBMK06,Wang2012,Lauret,AAYSMS14,Bohm,NAYN16,Arv2016}.

Parts of our work can be viewed as proposing a general criterion for detecting whether an ancient Ricci flow exists on a given Riemmanian manifold as well as a method for predicting what its limit should be. The notion of an ancient Ricci flow is 
central to the theory of geometric evolution equations. These ancient flows are the prototype for singularity models for the Ricci flow and have been crucial, for example, in Perelman's work. A basic question in the field is when an ancient Ricci flow exists on a given manifold. If it exists, one would like to describe its limit. In this article, we put forward a conjecture that stipulates that a time-reversed version of the Ricci iteration in fact detects (i.e., is a sufficient condition for) the existence of ancient Ricci flows. We prove this conjecture in the setting of certain homogeneous spaces. Moreover, we show that the limits of the ancient Ricci flows coincide with the limits of the time-reversed iterations. We hope this provides motivation for considering this dynamical approach in general. One can also imagine this being a useful idea for other evolution equations and dynamical systems.

As we find in this article, 
in the non-K\"ahler setting the Ricci iteration
exhibits new types of behavior as compared to the K\"ahler setting. For instance, solutions to~\eqref{RIEq} 
may not exist for all $i\in\NN$.
In addition,  existence (and hence convergence) may fail 
even when $M$ admits a homogeneous Einstein metric. 
Finally, and perhaps surprisingly, we construct the first non-trivial {\it ancient Ricci iterations}. These are sequences of Riemannian metrics
$\{g_1,g_0,g_{-1},g_{-2},\ldots\}$ such that
\begin{equation}
\begin{aligned}
\label{RRIEq}
g_{i-1}=\Ric g_{i}, \qq i=1,0,-1,-2,\ldots.
\end{aligned}
\end{equation}
Thus, on the Riemannian manifold $(M,g_1)$, the time-reversed Ricci iteration 
$$
g_1,\; \Ric g_1,\; \Ric\Ric g_1,\; \ldots
$$ 
exists.
These provide examples of
Riemannian metrics to which the Ricci operator can be applied infinitely many times, answering questions posed in~\cite{R08}. We find such metrics in dimensions as low as~6.

To put ancient Ricci iterations in perspective,
as explained in  \cite[\S3]{R08}, the Ricci iteration  \eqref{RIEq} can be thought of as the backward Euler discretization of the normalized Ricci flow, and hence one may expect good analytical properties for it. On the other hand, in doing the reverse procedure, i.e., iterating the Ricci operator, one typically loses two derivatives in each step. What is more, the Ricci operator can only be iterated as long as positivity is preserved. Therefore, at least from the analytical point of view, it seems unlikely that such ancient iterations should exist. 
We show they are closely tied to the geometry, exist in tandem with ancient solutions of the Ricci flow
constructed by Buzano~\cite{MB14}, and exhibit a relative compactness property similar to that of the Ricci iteration.
In some cases, they emerge from one Einstein metric as $i\ra-\infty$ and converge to a second Einstein metric
as $i\ra\infty$. In other cases, they emerge from a collapsed Einstein limit $G/K$ where 
$H\subsetneq K \subsetneq G$
as $i\ra-\infty$ and converge to an Einstein metric
as $i\ra\infty$. We find examples of the latter phenomenon starting in dimension~7. 

One of the reasons the Ricci iteration is relatively well-understood in the
K\"ahler setting is that the prescribed Ricci curvature equation reduces
in that setting to a complex Monge--Amp\`ere equation. The existence and uniqueness are then given by the Calabi--Yau theorem~\cite{Calabi,Yau}. More precisely, there always exist unique \K metrics solving~\eqref{RIEq} for $i\in\NN$, provided that
$g_1$, and hence each $g_i$, represents ($2\pi$ times) the first Chern class. 
The first result on the Ricci iteration established smooth convergence to a \KE metric under some symmetry assumptions~\cite[Theorem 3.3]{R08} that turn out to hold in a number of interesting cases \cite{CS,CS2,CW,CPS}. A conjecture
 stipulates that the Ricci iteration should in fact always converge in an appropriate sense to a K\"ahler--Einstein metric when one exists~\cite [Conjecture 3.2]{R08}. Recent progress in this direction is due to~\cite [Theorem~A]{BBEGZ}, \cite[Corollary 3.10]{DR}.

Clearly, the first obstacle in the general Riemannian setting is to understand the prescribed Ricci curvature equation.  
While this equation has been the subject of active research (e.g., 
\cite[Chapter~5]{AB87}, \cite{RH84,DDT85,DDTHG99,ED02,PhD03,AP13b,RPLAMP15,GS16})---when compared 
to the K\"ahler case---it is understood rather poorly on general Riemannian manifolds. Recent results of the first-named author~\cite{AP15} provide a replacement for the Calabi--Yau theorem on certain classes of homogeneous spaces and are crucial in the present article.

The second obstacle is to show convergence. In order to do so, one may seek and analyze various monotonic quantities associated with the iteration. 
One key ingredient in the proofs of our Theorems~\ref{s2Thm} and~\ref{AncientThm} is the monotonicity of the quantity $\alpha_i$ given by~\eqref{alphaiEq}. 
Analogous monotonicity was exploited by Buzano in 
the analysis of the Ricci flow~\cite{MB14}. While the iteration turns out to exhibit similar limiting behavior to the flow in many situations, there are cases where they differ; see, e.g., Remark~\ref{DiffRFRIRemark}.

The third, and perhaps new, 
obstacle in the general Riemannian setting is 
that---as we find in this article---the Ricci iteration can exhibit a range of behaviors depending on the homogeneous space {\it and } the starting point, quite in contrast to the  K\"ahler setting:
there the Ricci iteration always exists and  either converges or not depending only on whether a K\"ahler--Einstein metric exists or not (regardless of the initial condition).
We will see that neither of these behaviors 
persists in the setting of homogeneous spaces. 
This is, perhaps, indicative of the relation between the Ricci iteration and the Ricci flow. Indeed,
the behavior of the Ricci flow on homogeneous spaces 
is rather intricate as compared to the  K\"ahler--Ricci flow on Fano manifolds (that always converges to a K\"ahler--Einstein metric when such a metric exists).

Finally, we remark that the Ricci iteration seems harder to understand than the Ricci flow.
This is certainly the case in the K\"ahler setting; perhaps, one reason is the wide variety of tools available
to analyze parabolic flows. 
Currently, a complete theory of the Ricci iteration for all compact homogeneous manifolds
seems out of reach, as even Einstein metrics and the Ricci flow are not completely understood in this setting. The study we carry out in the present article---while restricted to a certain class of highly symmetric manifolds in which the Ricci iteration reduces to a sequence of systems of algebraic equations---at least gives a setting in which the Ricci iteration can be 
fully understood and compared to the Ricci flow in detail. Among other things, it suggests what kind of phenomena to expect in general. 

\section{Main results}
\label{sec_results}

Let $H $ be  a closed connected subgroup of a compact connected Lie group $G$.
In this article, we focus on the compact homogeneous space
$$
M:=G/H
$$
and $G$-invariant Riemannian metrics on it. 

\subsection{Preliminaries}
\label{prelim_subsec}

The group $G$ acts on $M$ by associating to $\lambda\in G$ the diffeomorphism
$L_\lambda:M\ra M$ defined by
$\nu H\mapsto \lambda \nu H$ for $\nu\in G$.
The action is called effective if the equality $L_\lambda=\h{id}$ implies that $\lambda$ is the identity of $G$. Henceforth, $G$ will be assumed to act effectively. 
This entails no loss of generality~\cite[\S7.12]{AB87}.

Suppose 
\begin{equation}
\begin{aligned}
\label{dimEq}
n:=\dim M\ge 3.
\end{aligned}
\end{equation}
Let $\mathfrak g$ denote the Lie algebra of $G$, and let $\Ad_G$ be the adjoint representation of $G$ on $\mathfrak g$. We fix a bi-invariant Riemannian metric on $G$. It induces an $\Ad_G(G)$-invariant
inner product $Q$ 
on $\mathfrak g$ and a 
$G$-invariant Riemannian metric $\hat g$ on $M$.
Choose a point $\mu\in M$. Then 
$H$ can be identified with $G\cap\h{Iso}_\mu(M,\hat g)$, where $\h{Iso}_\mu(M,\hat g)$ is the isotropy group of $(M,\hat g)$ at $\mu$.

Let  $\mathfrak h$ denote the Lie algebra of $H$. 
We know $H$ is compact (as $G$ is), and so is 
$\Ad_G(H)$. The $Q$-orthogonal complement of $\mathfrak h$ in $\mathfrak g$ is an $\Ad_G(H)$-invariant subspace of $\mathfrak g$. We denote this subspace by $\mathfrak m$. Thus, 
\begin{equation*}
\begin{aligned}
\mathfrak g=\mathfrak m\oplus \mathfrak h.
\end{aligned}
\end{equation*} 
We naturally identify
$$
\mathfrak m\cong \mathfrak g/\mathfrak h \cong T_\mu M.
$$ 
Every $G$-invariant Riemannian metric on $M$ induces an $\Ad_G(H)$-invariant inner product on $\mathfrak m$. The converse also holds \cite[Theorem~7.24]{AB87}, so
\begin{align}
\label{calMEq}
\calM &:=\{\h{$G$-invariant Riemannian metrics on $M$}\}
\cr
&\cong\{\h{$\h{Ad}_G(H)$-invariant inner products on $\mathfrak m$}\}.
\end{align}
Unlike in the case of K\"ahler metrics on a K\"ahler manifold, the space of $G$-invariant metrics on $M$ is finite-dimensional. Nevertheless, we will see that
the Ricci iteration exhibits more complicated phenomena than in the \K setting, where the space of \K metrics is infinite-dimensional.

Consider a $Q$-orthogonal $\Ad_G(H)$-invariant decomposition
\begin{align}\label{m_decomp}
\mathfrak m=\mathfrak m_1\oplus\cdots\oplus\mathfrak m_s
\end{align}
such that $\Ad_G(H)|_{\mathfrak m_i}$ is irreducible for each $i=1,\ldots,s$. Let
\begin{equation}
\begin{aligned}
\label{diEq}
d_i:=\dim{\mathfrak m_i}\in\NN.
\end{aligned}
\end{equation} 
While the space $\mathfrak m$ may admit more than one decomposition of the form~\eqref{m_decomp},
the number of summands, $s$, is determined by $G$ and $H$, i.e., it is
the same for all such decompositions. On the other hand, 
the summands $\mathfrak m_1,\ldots,\mathfrak m_s$ {\it are} determined
uniquely up to order if 
\begin{equation}
\begin{aligned}
\label{InEquivAssumEq}
\h{$\Ad_G(H)|_{\mathfrak m_i}$ is inequivalent to $\Ad_G(H)|_{\mathfrak m_k}$ whenever $i\ne k$.}
\end{aligned}
\end{equation}
Indeed, suppose $\mathfrak m_1'\oplus\cdots\oplus\mathfrak m_{s'}'$ is another decomposition. 
The $Q$-orthogonal projection from $\mathfrak m_i$ to $\mathfrak m_j'$ is an equivariant map. Therefore, it is either an isomorphism or zero by Schur's lemma. We will assume~\eqref{InEquivAssumEq} holds in order to state our first two results. For our other results, however, we will {\it not} assume it. 

The isotropy representation of $M$ is said to be \emph{monotypic} if it satisfies~\eqref{InEquivAssumEq}.

Given $T\in\mathcal M$, it is always possible to choose the decomposition~\eqref{m_decomp} so that $T$ has a simple ``diagonal" form.
More precisely, 
let 
$$
\pi_i:\mathfrak m\ra \mathfrak m_i
$$ denote the natural projections
induced by~\eqref{m_decomp}. Recall that
$\Ad_G(H)|_{\mathfrak m_i}$ is irreducible for each $i$. There exists a choice of~\eqref{m_decomp} such that $T$ has the form
\begin{align}\label{T_def}
T=\sum_{i=1}^s z_i\pi_i^* Q,\q z_i>0,
\end{align}
see~\cite[page~180]{MWWZ86}. Equality~\eqref{T_def} can also be written as
$$
T(X,Y)=\sum_{i=1}^s z_iQ(\pi_i(X),\pi_i(Y)),\qquad X,Y\in\mathfrak m.
$$

As~\eqref{calMEq} shows, we can identify $\mathcal M$ with a subset of $\mathfrak m^*\otimes\mathfrak m^*$. With this identification, the topology of $\mathfrak m^*\otimes\mathfrak m^*$ induces a topology on $\mathcal M$. Convergence in $\mathcal M$ is equivalent to the smooth convergence of Riemannian metrics on $M$.

In what follows, we say $H$ is maximal in $G$ if $H$ is a maximal connected Lie subgroup of $G$ or, equivalently, if $\mf h$ is
a maximal Lie subalgebra of $\mf g$.

When $s=1$, the space $\mathcal M$ is one-dimensional, and so we will always assume $s>1$.
(We remark that our earlier assumption  \eqref{dimEq}  is reasonable since the
only compact two-dimensional manifolds admitting positively-curved metrics, $S^2=SO(3)/SO(2)$
and $\mathbb RP^2=O(3)/(O(2)\times O(1))$, satisfy $s=1$.) 
Theorems~\ref{s2Thm} and~\ref{AncientThm} below concern the case when $s=2$ and~\eqref{InEquivAssumEq} holds. To formulate them, denote
\begin{equation}
\begin{aligned}
\label{EEq}
\calE:=\{\h{Einstein metrics in $\calM$}\}.
\end{aligned}
\end{equation}
If $\mathcal E$ is nonempty, define
\begin{align}
\alpha_-&:=
\inf\big\{{z_1}/{z_2}\,:\,
z_1,z_2~\mbox{satisfy~\eqref{T_def} for some}~T\in\mathcal E\big\},\notag
\\
\alpha_+&:=
\sup\big\{{z_1}/{z_2}\,:\,z_1,z_2~\mbox{satisfy~\eqref{T_def} for some}~T\in\mathcal E\big\}.\label{alphapmEq}
\end{align}
Lemma~\ref{RicPosEinLemma} below implies that, when $s=2$ and~\eqref{InEquivAssumEq} holds, there are only finitely many Einstein metrics in $\calM$ up to scaling; therefore, $\alpha_-$ and $\alpha_+$ lie in $(0,\infty)$.

Theorems~\ref{s2Thm} and~\ref{AncientThm} distinguish between the maximal and non-maximal cases. 
When $H$ is maximal in $G$, at least one Einstein metric exists in $\mathcal M$~\cite[p. 177]{MWWZ86}.
When $H$ is not maximal, there is a connected Lie subgroup $K$
of $G$ such that
\begin{equation}
\begin{aligned}
\label{HKGEq}
H\subsetneq K \subsetneq G,
\end{aligned}
\end{equation}
and we write $\mathfrak k$ for the Lie algebra of $K$. If $s=2$ and~\eqref{InEquivAssumEq} is satisfied,
we assume without loss of generality that
\begin{align}\label{km1hEq}
\mathfrak k=\mathfrak h\oplus\mathfrak m_1.
\end{align}

Theorems~\ref{s2Thm} and~\ref{AncientThm} will require that the projection $\pi_1[\mathfrak m_2,\mathfrak m_2]$ be nonzero, i.e.,
\begin{equation}
\begin{aligned}
\label{QHypEq}
Q([X,Y],Z)\ne0 \h{\ for some $X\in\mathfrak m_1$ and $Y,Z\in\mathfrak m_2$}.
\end{aligned}
\end{equation}
This hypothesis automatically holds when $H$ is maximal in $G$
(see Lemma \ref{gammaMaxLemma}). If an intermediate subgroup $K$ satisfying~\eqref{HKGEq} and~\eqref{km1hEq} exists, formula~\eqref{QHypEq}
may fail to hold. However, when~\eqref{QHypEq} is false, all the $G$-invariant metrics on $M$ have the same Ricci curvature (see~\eqref{SameRicEq}). In this case, there is at most one Ricci iteration in~$\mathcal M$. All the metrics in this iteration must be the same. Moreover, they must be Einstein.

If $s=2$ and $T\in\mathcal M$ is given by~\eqref{T_def}, then the ratio $z_1/z_2$ determines $T$ up to scaling. Since we are assuming $z_1,z_2>0$, one may think of the space of such ratios as the moduli space of $G$-invariant metrics on~$M$.

\subsection{Two irreducible isotropy summands}
\label{}

Our first result completely describes the Ricci iteration on compact homogeneous spaces with two inequivalent isotropy summands. We refer to~\cite{WDMK08,CH12} for a classification of such spaces; see also Examples~\ref{example_nontriv} and~\ref{example_triv} below. The Ricci flow on them was studied in~\cite{AC2011,LGRM12,MB14}.

\begin{theorem}[Ricci iteration for $s=2$]\label{s2Thm}
Let $M=G/H$. Assume that $s = 2$ in the decomposition~(\ref{m_decomp}) and that the isotropy representation of $M$ is monotypic (i.e.,~(\ref{InEquivAssumEq}) holds).
Let $T\in\mathcal M$ (recall~(\ref{calMEq})) be 
an arbitrary $G$-invariant metric
given by~(\ref{T_def}).

\begin{enumerate}[(i)]
\item
Suppose $H$ is maximal in $G$. There exists a unique sequence $\{g_i\}_{i\in\NN}\subset\mathcal M$ such that the iteration equation~(\ref{RIEq}) holds for all $i\in\mathbb N$ and $g_1=cT$ for some $c>0$. The metrics $\{g_i\}_{i\in\NN}$ converge smoothly to an Einstein metric on $M$ as $i$ tends to $\infty$.

\item
Suppose $G$ has a connected Lie subgroup $K$ satisfying~(\ref{HKGEq}) and~(\ref{km1hEq}). Let the projection $\pi_1[\mathfrak m_2,\mathfrak m_2]$ be nonzero (i.e., let~(\ref{QHypEq}) hold).

\begin{enumerate}

\item
Assume $\Ad_G(H)|_{\mathfrak m_1}$ is trivial. Then $\calE\ne\emptyset$.
There exists a unique sequence $\{g_i\}_{i\in\NN}\subset\mathcal M$ 
such that the iteration equation~(\ref{RIEq}) holds
for all $i\in\mathbb N$ and $g_1=cT$ for some $c>0$. This sequence converges smoothly to the unique 
$G$-invariant Einstein metric on $M$ with Einstein constant 1.

\item
Assume $\Ad_G(H)|_{\mathfrak m_1}$ is nontrivial
and $\mathcal E\ne\emptyset$. A sequence $\{g_i\}_{i\in\NN}\subset\mathcal M$ 
satisfying~(\ref{RIEq}) for all $i\in\mathbb N$ and $g_1=cT$ for some~$c>0$ exists if and only if
the inequality $z_1/z_2\ge\a_-$ holds.
When it exists, the sequence $\{g_i\}_{i\in\NN}$ is unique and converges smoothly to an Einstein metric on $M$.

\item If  $\mathcal E=\emptyset$,
there is no sequence $\{g_i\}_{i\in\NN}\subset\mathcal M$ 
such that~(\ref{RIEq}) holds
for all $i\in\mathbb N$.
\end{enumerate}
\end{enumerate}
\end{theorem}

\begin{remark} {\rm
\label{DiffRFRIRemark}
The Ricci iteration and the Ricci flow exhibit essentially the same limiting behavior except in two cases:
in the case (ii-b) with initial condition satisfying $z_1/z_2<\a_-$ and in the case (ii-c), the flow converges to a collapsed Einstein metric on $G/K$ \cite[Theorems~3.4,~3.8]{MB14}, while the iteration stops after finitely many steps.
} \end{remark}

One motivation for introducing the Ricci iteration in~\cite{R08} was a question of DeTurck and Nadel
\cite[Remark 4.63]{RThesis}, \cite{Nadel}: 
does the Ricci operator possess nontrivial (i.e., non-Einstein) fixed points
(i.e., solutions to $\Ric\Ric\ldots\Ric g=g$)? 
The proof of Theorem \ref{s2Thm} answers this question in the negative in the setting of the theorem since we exhibit a strictly monotone quantity along the iteration. 

Another natural question, posed in \cite{R08}, is, roughly, whether there exist discrete analogs of ancient Ricci flows.
More precisely, given a Riemannian metric $g$, one may associate to it a Riemannian invariant 
\begin{align*}
r(g):=\sup\big\{k\in\NN\,:\,\underbrace{\Ric\Ric \ldots\Ric}_{k-1~\rm{times}} g~\mbox{is a Riemannian metric}\big\}
\end{align*}
called the Ricci index of $g$ \cite[\S10.7]{R08}.
Following \cite[Definition~10.21]{R08} one can then define a filtration of~$\calM$
\begin{align}\label{filtEq}
\calM=\calM^{(1)}\supset \calM^{(2)}\supset
\cdots\supset \calM^{(l)}\supset \cdots,
\end{align}
by setting 
$$
\calM^{(l)}=\{g\in\calM\,:\, r(g)\ge l\}. 
$$

\begin{question}[{\rm \cite[p. 1562]{R08}}]
\label{MinfQ}
What is $\calM^{(\infty)}:=\cap_{l=1}^\infty \calM^{(l)}$? What is the relation between 
$\calM^{(\infty)}$ and the Ricci flow?
\end{question}

Clearly, $\calE\subset \calM^{(\infty)}=\{g\in\calM\,:\,r(g)=\infty\}$. But, do there exist $g$ with $r(g)=\infty$ that are neither Einstein nor direct sums of Einstein metrics?
We answer this last question in the affirmative. We also answer Question \ref{MinfQ} 
on compact homogeneous spaces with $s=2$ satisfying~\eqref{InEquivAssumEq}. 
The interest in metrics with infinite Ricci index is that one may use them as starting points for the 
ancient Ricci iteration~\eqref{RRIEq}.
As pointed out to us by M. Wang after the completion of this article,
Walton \cite{Walton} proved Theorem~\ref{AncientThm} below with the exception of part 
(ii-a) recently and independently.
Comparing the result to the study of the Ricci flow by Buzano~\cite{MB14}
shows that the ancient Ricci iterations exist in tandem with ancient Ricci flows and have similar limits.

\begin{theorem}[ancient Ricci iterations for $s=2$]\label{AncientThm}
Let $M=G/H$. Assume that $s = 2$ in the decomposition~(\ref{m_decomp}) and that the isotropy representation of $M$ is monotypic (i.e.,~(\ref{InEquivAssumEq}) holds). 
In the statements below, 
we always assume $T\in\mathcal M$ (recall~(\ref{calMEq})) 
is given by~(\ref{T_def}).

\begin{enumerate}[(i)]
\item
Suppose $H$ is maximal in $G$. 
Then (recall~(\ref{alphapmEq}))
\begin{equation}
\begin{aligned}
\label{Minfmax}
\calM^{(\infty)}=\{T\in\calM\,:\, z_1/z_2\in[\a_-,\a_+]\}.
\end{aligned}
\end{equation}
Whenever $r(g_1)=\infty$, the metrics 
$\{g_{-i}\}_{i=-1}^\infty$
given by the ancient iteration equation~(\ref{RRIEq}) converge smoothly 
to an Einstein metric on $M$.

\item
Suppose $G$ has a connected Lie subgroup $K$ satisfying~(\ref{HKGEq}) and~(\ref{km1hEq}). Let the projection $\pi_1[\mathfrak m_2,\mathfrak m_2]$ be nonzero (i.e., let~(\ref{QHypEq}) hold).
\begin{enumerate}

\item
Assume $\Ad_G(H)|_{\mathfrak m_1}$ is trivial.
Then
\begin{align}
\label{Minfminus}
\calM^{(\infty)}=
\{T\in\calM\,:\, z_1/z_2\le\a_-\}.
\end{align}
Whenever $r(g_1)=\infty$ and $g_1$ is non-Einstein, the metrics 
$\{g_{-i}\}_{i=-1}^\infty$
given by~(\ref{RRIEq}) converge smoothly 
to a~degenerate tensor that is the pullback of an Einstein metric $g_E$ on $G/K$ under the inclusion map $G/K\hookrightarrow G/H$.
The manifolds $(M,g_{-i})$ converge in the Gromov--Hausdorff topology to $(G/K,g_E)$ as~$i\to\infty$.

\item 
Assume $\Ad_G(H)|_{\mathfrak m_1}$ is nontrivial. The set
$\calM^{(\infty)}$ is empty if $\calE$ is, and
\begin{equation}
\begin{aligned}
\label{Minfnonmax}
\calM^{(\infty)}
=
\{T\in\calM\,:\,  z_1/z_2\le \a_+\}.
\end{aligned}
\end{equation}
otherwise. Whenever $r(g_1)=\infty$, the metrics 
$\{g_{-i}\}_{i=-1}^\infty$
given by the ancient iteration equation~(\ref{RRIEq}) converge smoothly 
to an Einstein metric on $M$.
\end{enumerate}
\end{enumerate}
\end{theorem}

In particular, $\calM^{(\infty)}\not=\calE$ always in~\eqref{Minfminus} and~\eqref{Minfnonmax}, as well as 
when $\a_-<\a_+$  (i.e., when there exist at least two distinct Einstein
metrics of volume 1) in \eqref{Minfmax}. We also observe that in the maximal case
(in fact, even when $s\ge2$ and~\eqref{InEquivAssumEq} does not necessarily hold)
\begin{equation}
\begin{aligned}
\label{decreasingEq}
\calM^{(l)}\setminus \calM^{(l+1)}
=\{g\in\calM\,:\, r(g)=l\}\not=\emptyset, \q l\in\NN,
\end{aligned}
\end{equation}
as can be seen by applying Theorem~\ref{max_uniqThm} $l$ times starting with
a degenerate positive-semidefinite nonzero~$T$. 
To compare with the K\"ahler setting, note that there also
the filtration of the space of K\"ahler metrics similar to~\eqref{filtEq} is strictly decreasing in the sense of \eqref{decreasingEq}
(by applying the Calabi--Yau theorem $l$ times starting with a non-positive form); however, a characterization of the analogue of $\calM^{(\infty)}$ is still missing.

In general, we make the following conjecture.

\begin{conj}\label{conj_anc}
A manifold admits an ancient Ricci iteration only if it admits an ancient Ricci flow.
\end{conj}

In other words, the conjecture gives a sufficient condition for the existence of ancient Ricci flows (on manifolds that admit positive Ricci curvature). 

Theorem~\ref{AncientThm} and the results of Buzano~\cite{MB14} show that this conjecture holds in the case of compact homogeneous spaces satisfying  $s=2$ and~\eqref{InEquivAssumEq}.  
In a forthcoming paper we study this conjecture on certain Lie groups and spheres, including the case of $S^3$, equipped with left-invariant metrics (note that it is easy to see,
as was also pointed out to us by W. Ziller after the completion of the present article, that some Berger metrics on $S^3$ have infinite Ricci index).

\subsection{Relative compactness for maximal isotropy}

Our next result concerns the case when $s\ge 2$
and the isotropy summands are allowed to be equivalent (i.e., 
when \eqref{InEquivAssumEq} does not 
necessarily hold) but when $H$ is assumed to be maximal. It is generally unknown whether, in this situation, the prescribed Ricci curvature problem has unique solutions in $\calM$. Therefore, it is not clear whether to expect convergence of the Ricci iteration or not. 
If indeed there is no uniqueness, then relative compactness is a reasonable replacement (cf.~\cite[Theorem~A]{BBEGZ}, \cite[Corollary 3.10]{DR} in 
the K\"ahler setting). We prove it in Theorem~\ref{GenThm}~(i). It turns out that a similar result holds for ancient Ricci iterations. This is the content of Theorem~\ref{GenThm}~(ii). If a sequence $\{g_i\}_{i\in\mathbb N}$ satisfying~\eqref{RIEq} or a sequence $\{g_{-i}\}^\infty_{i=-1}$ satisfying~\eqref{RRIEq} converges to some tensor field,
this limit may fail to be positive-definite, i.e., a Riemannian metric. When $H$ is maximal, this is excluded by Corollary~\ref{cor_subseq}.

\begin{theorem}[existence and relative compactness for $s\ge 2$]\label{GenThm}
Let $M=G/H$. Suppose that $s \ge 2 $ in the decomposition~(\ref{m_decomp}) and that $H$ is maximal in $G$. 

\begin{enumerate}[(i)]
\item Given $T\in\mathcal M$, there exist a sequence $\{g_i\}_{i\in\NN}\subset\mathcal M$ 
satisfying the iteration equation~(\ref{RIEq}) for all $i\in\NN$ and $g_1=cT$ for some $c>0$. Any such sequence is relatively compact in $\mathcal M$.

\item
Any sequence $\{g_{-i}\}^\infty_{i=-1}\subset\calM$ satisfying the ancient iteration equation~(\ref{RRIEq}) is relatively compact in~$\calM$.
\end{enumerate}
\end{theorem}

\begin{remark} {\rm
\label{}
The existence part of Theorem~\ref{GenThm}~(i) becomes false if the maximality condition on $H$ is removed. Indeed, Theorem~\ref{s2Thm}~(ii-b)--(ii-c) demonstrates that it may be impossible to find a sequence $\{g_i\}_{i\in\NN}$ satisfying~\eqref{RIEq}. If $\{g_i\}_{i\in\NN}$ is known to exist and~$H$ is not maximal, whether or not $\{g_i\}_{i\in\NN}$ must be relatively compact in $\mathcal M$ is an open question.
} \end{remark}

\begin{remark} {\rm
The collapsing construction of Theorem~\ref{AncientThm}~(ii-a) shows that none of the results for ancient Ricci iterations appearing in this subsection holds if the maximality assumption is dropped.}
\end{remark}

We state two consequences of Theorem~\ref{GenThm}.

\begin{corollary}\label{cor_subseq}
Let $M=G/H$. Suppose that $s \ge 2 $ in the decomposition~(\ref{m_decomp}) and that $H$ is maximal in $G$. 
Any sequence $\{g_i\}_{i\in\mathbb N}\subset\mathcal M$ satisfying the iteration equation~(\ref{RIEq})
(or $\{g_{-i}\}^\infty_{i=-1}\subset\calM$ satisfying the ancient iteration equation~(\ref{RRIEq}))
has a subsequence converging smoothly to some Riemannian metric in~$\mathcal M$.
\end{corollary}

\begin{proof}
Theorem~\ref{GenThm} implies that the closures of $\{g_i\}_{i\in\mathbb N}$ and $\{g_{-i}\}^\infty_{i=-1}$ in $\mathcal M$ are compact subsets of $\mathcal M$. Therefore, these sequences contain subsequences whose limits lie in $\mathcal M$.
\end{proof}

Similarly to~\eqref{calMEq}, there exists an isomorphism
\begin{align}\label{barMEq}
\mathcal T
&:=\{\h{$G$-invariant symmetric (0,2)-tensor fields on $M$}\}\notag
\\
&\cong\{\h{$\h{Ad}_G(H)$-invariant symmetric bilinear forms on $\mathfrak m$}\}.
\end{align}
Thus, we can identify $\mathcal T$ with a subset of $\mathfrak m^*\otimes\mathfrak m^*$. The topology of $\mathfrak m^*\otimes\mathfrak m^*$ induces a topology on~$\mathcal T$, as it did on $\mathcal M$.

\begin{corollary}\label{cor_Einstein}
Suppose that $M=G/H$, that $s \ge 2 $ in the decomposition~(\ref{m_decomp}), and that $H$ is maximal in~$G$. 
Assume a sequence $\{g_i\}_{i\in\mathbb N}\subset\mathcal M$ satisfying the iteration equation~(\ref{RIEq}) 
(or $\{g_{-i}\}^\infty_{i=-1}\subset\calM$ satisfying the ancient iteration equation~(\ref{RRIEq}))
converges smoothly to some tensor field $g_\infty\in\mathcal T$. Then $g_\infty$
is an Einstein metric.
\end{corollary}

\begin{proof}
Assume $\{g_i\}_{i\in\mathbb N}$ converges to $g_\infty$. The other case is treated analogously. 
According to  Corollary~\ref{cor_subseq}, $\{g_i\}_{i\in\mathbb N}$ contains a subsequence that converges to some Riemannian metric in $\mathcal M$. This metric must coincide with $g_\infty$ by uniqueness of limits. Thus, $g_\infty$ lies in $\mathcal M$. To see that $g_\infty$ satisfies the Einstein equation, one simply needs to pass to the limit in equality~\eqref{RIEq} as $i$ tends to~$\infty$.
\end{proof}

We propose the following conjecture. It is motivated by Theorems~\ref{s2Thm}~(i) and~\ref{GenThm}~(i), as well as the fact that compact homogeneous Ricci solitons are necessarily Einstein (see~\cite[Theorem~1.4]{MJ15}).

\begin{conj}
Suppose $H$ is maximal in $G$. Let $g_1$ be a Riemannian metric in $\calM$.
\begin{enumerate}[(i)]
\item If a sequence $\{g_i\}_{i\in\NN}\subset\calM$ starting with $g_1$ and satisfying the iteration equation~(\ref{RIEq}) exists and is unique, then this sequence converges to an Einstein metric.

\item
If $r(g_1)=\infty$, then the sequence $\{g_{-i}\}^\infty_{i=-1}\subset\calM$ 
given by the ancient iteration equation~(\ref{RRIEq}) converges to an Einstein metric.
\end{enumerate}
\end{conj}

\subsection{Organization}
\label{}

This article is organized as follows. Section~\ref{PRCSec} recalls facts on homogeneous metrics and the prescribed Ricci curvature equation from
\cite{MWWZ86,JSPYS97,AP15} that play a central r\^ole in our arguments.
Section~\ref{sec_2sum} contains the proofs of our results for $s=2$: 
Theorem~\ref{s2Thm}~(i) is established in~\S\ref{subsec_2max}, Theorem~\ref{s2Thm}~(ii)
in~\S\ref{subsec_non-max}, and Theorem~\ref{AncientThm} in~\S\ref{ancientSubSec}. 
We conclude Section~\ref{sec_2sum} by giving two explicit examples in~\S\ref{examSubSec}. In Section~\ref{GenSec} we treat the case $s\ge2$ with~$H$ maximal, proving Theorem~\ref{GenThm}.

\section{The prescribed Ricci curvature equation}
\label{PRCSec}

In this section we recall results on homogeneous metrics and the prescribed Ricci curvature problem, following \cite{MWWZ86,JSPYS97,AP15}. We begin with a theorem concerning the existence of solutions to the prescribed (positive-semidefinite) Ricci curvature equation~\cite [Theorem 1.1]{AP15}. The sign of the constant $c$ below is explained by Lemma~\ref{RicPosLemma} below.

\begin{theorem}[prescribed curvature for maximal isotropy]\label{max_uniqThm}
Let $H$ be maximal in $G$. Given a positive-semidefinite nonzero $T\in\mathcal T$, there exists a metric $g\in\mathcal M$ whose Ricci curvature coincides with $cT$ for some~$c>0$.
\end{theorem}

Our arguments will involve
three arrays of numbers
\begin{equation}
\begin{aligned}
\label{dataEq}
\{b_i\}_{i=1}^s,\q  
\{\gamma_{ik}^l\}_{i,k,l=1}^s, \q
\{\zeta_i\}_{i=1}^s, 
\end{aligned}
\end{equation}
associated with the inner product $Q$ and the decomposition~\eqref{m_decomp}. 
To define the first one, denote by $B$ the Killing form on the Lie algebra~$\mathfrak g$. As $\Ad_G(H)|_{\mathfrak m_i}$ is irreducible, there exists $b_i\ge0$ such that
\begin{align}\label{b_def}
B|_{\mathfrak m_i} = -b_iQ|_{\mathfrak m_i}.
\end{align}

Next, let $\Gamma_{ik}^l\in\mathfrak m_i^*\otimes\mathfrak m_k^*\otimes\mathfrak m_l$  
be the tensor
$$
\Gamma_{ik}^l(X,Y):=\pi_l([X,Y]), \q X\in \mathfrak m_i, Y\in \mathfrak m_k,
$$
and denote by $\gamma_{ik}^l$ the squared norm of $\Gamma_{ik}^l$ with respect to $Q$. Thus,
\begin{align}\label{gamma_def}
\gamma_{ik}^l:=
\sum Q([e_{\iota_i},e_{\iota_k}], e_{\iota_l})^2,
\end{align}
where $\{e_j\}_{j=1}^n$ is a $Q$-orthonormal basis  
of $\mathfrak m$ adapted to the decomposition~\eqref{m_decomp},
and where the sum is taken over all $\iota_i, \iota_k$ and $\iota_l$ such that $e_{\iota_i} \in \mathfrak m_i$, $e_{\iota_k} \in \mathfrak m_k$ and $e_{\iota_l} \in \mathfrak m_l$. 
The $\gamma_{ik}^l$ are often called the {\it structure constants} of $M$.
Since $Q$ is $\Ad_G(G)$-invariant, it follows that 
\begin{equation}
\begin{aligned}
\label{QAdInvEq}
Q([X,Y],Z)=Q(Y,[X,Z]),\q X,Y,Z\in\mathfrak g.
\end{aligned}
\end{equation}
It is, therefore, evident 
that $\gamma_{ik}^l$ is symmetric in all three indices. 

For the third piece of notation, fix a $Q$-orthonormal basis $\{w_j\}_{j=1}^q$
of the Lie algebra $\mathfrak h$. The irreducibility of $\Ad_G(H)|_{\mathfrak m_i}$ implies the existence of $\zeta_i\ge0$ such that the Casimir operator 
$$
C_{\mathfrak m_i,Q|_{\mathfrak h}}:=
-\sum_{j=1}^q\ad w_j\circ\ad w_j
$$
acting on $\mathfrak m_i$ satisfies
\begin{align}\label{zeta_def}
C_{\mathfrak m_i,Q|_{\mathfrak h}}
=\zeta_i\, \h{id}.
\end{align}

\begin{remark} {\rm
\label{zetaRemark}
If $\zeta_i$ vanishes, then $\Ad_G(H)$ is trivial on $\mathfrak m_i$ and the dimension of $\mathfrak m_i$ equals~1. To see this, simply observe that 
\begin{align*}
Q(C_{\mathfrak m_i,Q|_{\mathfrak h}}e_k,e_k)=-\sum_{j=1}^qQ([w_j,[w_j,e_k]],e_k)=-\sum_{j=1}^qQ([w_j,e_k],[w_j,e_k]).
\end{align*}
If~\eqref{zeta_def} holds with $\zeta_i=0$, the above formula implies $[w_j,e_k]=0$ for all $k$, which means $\ad w_j=0$ on~$\mathfrak m_i$. 
It follows that $\Ad_G(H)$ is trivial and $\mathfrak m_i$ is 1-dimensional by irreducibility.
Note also that $\zeta_1$ and $\zeta_2$ cannot vanish together when $s=2$ in~\eqref{m_decomp}. This would imply $\dim M=2$,
contradicting~\eqref{dimEq}.  
} \end{remark}

The constants  \eqref{dataEq} are related by the formula \cite[Lemma~(1.5)]{MWWZ86}
\begin{align}\label{Casimir}
b_i=2\zeta_i+\frac1{d_i}\sum_{k,l=1}^s\gamma_{ik}^l,
\end{align}
where $d_i$ is given by~\eqref{diEq}. The following well-known result is adapted from~\cite[Lemma~1.1]{JSPYS97}. We include the proof since our notation and 
assumptions are somewhat different (e.g., we do not assume $G$ is semisimple, and we work with 
$(0,2)$-type and not $(1,1)$-type tensors).

\begin{lemma}
\label{PRCLemma}
Suppose (\ref{InEquivAssumEq}) holds.
Let $g\in \calM$ be given by
\begin{align}\label{gEq}
g=\sum_{i=1}^s x_i\pi_i^* Q, \q x_i>0.
\end{align}
Then
\begin{equation}
\begin{aligned}
\label{RicgEq}
\Ric g=\sum_{i=1}^s  r_i\pi_i^* Q,
\end{aligned}
\end{equation}
where
\begin{align*}
r_i=\frac{b_i}{2}
+\sum_{j,k=1}^s\frac{\gamma_{jk}^i}{4d_i}
\Big(\frac{x_i^2}{x_jx_k}
-2\frac{x_j}{x_k}
\Big).
\end{align*}
\end{lemma}

\begin{proof}
Since $g$ is $G$-invariant, so is $\Ric g$, which 
means~\eqref{RicgEq} holds.
The $i$-th component of $\Ric g$ equals
\cite[Corollary~7.38]{AB87} 
\begin{align}
\label{rilastEq}
r_i=-\frac12B(X,X)-\frac12\sum_{j=1}^ng([X,e_j]_\mathfrak m,[X,e_j]_\mathfrak m)+\frac14\sum_{j,k=1}^ng([e_j,e_k]_\mathfrak m,X)^2,
\end{align}
where $X\in\mathfrak m_i$ is any vector such that $Q(X,X)=1$, the subscript $\mathfrak m$ denotes the $g$-orthogonal projection onto $\mathfrak m$, and $\{e_j\}_{j=1}^n$ is a $g$-orthonormal basis of $\mathfrak m$ adapted to the decomposition~\eqref{m_decomp}. Note that the vector $Z$ in 
\cite[Corollary~7.38]{AB87} is identically zero for $G$ admitting a bi-invariant Haar measure
\cite[p. 184]{AB87} (in particular for $G$ compact).

Define
$$I^{(i)}:=\{\iota\,:\,e_{\iota}\in\mathfrak m_i\}.$$ 
In formula~\eqref{rilastEq}, set 
$
X=\sqrt{x_i}e_{\iota_i}
$
for some $\iota_i\in I^{(i)}$. 
By~\eqref{b_def},
\begin{align*}
r_i&=\frac{b_i}2g(e_{\iota_i},e_{\iota_i})-\frac12\sum_{j=1}^nx_ig([e_{\iota_i},e_j]_\mathfrak m,[e_{\iota_i},e_j]_\mathfrak m)+\frac14\sum_{j,k=1}^nx_ig([e_j,e_k]_\mathfrak m,e_{\iota_i})^2 \\ 
&=\frac{b_i}2-\frac12\sum_{j,k=1}^nx_ig( g([e_{\iota_i},e_j],e_k)e_k, g([e_{\iota_i},e_j],e_k)e_k) 
+\frac14\sum_{j,k=1}^nx_ig([e_j,e_k],e_{\iota_i})^2
\\ 
&=\frac{b_i}2-\frac12\sum_{j,k=1}^nx_ig([e_{\iota_i},e_j], e_k)^2+\frac14\sum_{j,k=1}^nx_i^3Q([e_j,e_k],e_{\iota_i})^2.
\end{align*}
Hence,
\begin{align*}
d_ir_i&=\sum_{\iota_i\in I^{(i)}}r_i=
\frac{d_ib_i}2-\frac12\sum_{\iota_i\in I^{(i)}}\sum_{j,k=1}^nx_ig([e_{\iota_i},e_j], e_k)^2+\frac14\sum_{\iota_i\in I^{(i)}}\sum_{j,k=1}^nx_i^3Q([e_j,e_k],e_{\iota_i})^2.
\end{align*}
Observe that the vectors $\{\tilde e_j\}_{j=1}^n$ defined by 
$$
\tilde e_{\iota_l}=\sqrt{x_l}e_{\iota_l},\qquad \iota_l\in I^{(l)},~l\in\{1,\ldots,s\},
$$
form a $Q$-orthonormal basis of $\mathfrak m$. Thus, 
\begin{align*}
r_i&=\frac{b_i}2-\frac1{2d_i}\sum_{j,k=1}^s\frac{x_k}{x_j} \sum Q([\tilde e_{\iota_i},\tilde e_{\iota_j}],\tilde e_{\iota_k})^2+\frac1{4d_i}\sum_{j,k=1}^s\frac{x_i^2}{x_jx_k}\sum Q([\tilde e_{\iota_j},\tilde e_{\iota_k}],\tilde e_{\iota_i})^2
\\ 
&=\frac{b_i}2-\frac1{2d_i}\sum_{j,k=1}^s\frac{x_k}{x_j}\gamma_{ij}^k+\frac1{4d_i}\sum_{j,k=1}^s\frac{x_i^2}{x_jx_k}\gamma_{ij}^k,
\end{align*}
where the sums without limits are taken over all $\iota_i\in I^{(i)}$, $\iota_j\in I^{(j)}$ and $\iota_k\in I^{(k)}$.
\end{proof}

The next lemma is a special case of a claim from the proof of~\cite[Theorem~(2.2)]{MWWZ86}. Recall that we assume $s>1$ in the decomposition~\eqref{m_decomp}.

\begin{lemma}
\label{gammaMaxLemma}
Suppose that $H$ is maximal in~$G$. Then, for any $i\in\{1,\ldots,s\}$, there exists some $k\in\{1,\ldots,s\}\setminus\{i\}$
such that
\begin{equation*}
\begin{aligned}
\gamma_{ii}^k>0.
\end{aligned}
\end{equation*}
\end{lemma}

\begin{proof}
We give the proof for the case $s=2$ which generalizes in an obvious way.
The assumption is equivalent to $\mathfrak h$ being a maximal Lie subalgebra of
$\mathfrak g$. In particular, $\mathfrak h\oplus\mathfrak m_1$ cannot be a 
Lie subalgebra of
$\mathfrak g$, so $\pi_2([\mathfrak m_1, \mathfrak m_1])\not=\{0\}$, i.e., $\gamma_{11}^2\not=0$.
Similarly, $\gamma_{22}^1\not=0$ since $\mathfrak h\oplus\mathfrak m_2$ cannot be a Lie subalgebra of
$\mathfrak g$.
\end{proof}

The next result implies that any $G$-invariant metric with definite Ricci curvature 
must actually have positive Ricci curvature if $H$ is maximal.

\begin{lemma}
\label{RicPosLemma}
Suppose that $g\in\mathcal M$ satisfies $\Ric g=cT$ for some $c\in\mathbb R$ and $T\in\mathcal M$. Then $c\ge0$. If $H$ is maximal in $G$, then $c>0$.
\end{lemma}

\begin{proof}
If $c<0$, Bochner's theorem \cite[Theorem~1.84]{AB87} implies that there are no Killing fields, so $G$
is trivial and $M$ is a point. By the same theorem, if $c=0$, then the connected component of the identity in the isometry group of $(M,g)$ is a torus. Since every Lie subgroup of a torus is abelian, $G$ must be abelian. 
Thus, $\gamma_{ij}^k=0$ for all $i,j,k$, contradicting Lemma \ref{gammaMaxLemma}.
\end{proof}

Whenever $s=2$ and $H$ is not maximal in $G$,
the prescribed (positive-semidefinite) Ricci curvature equation is completely understood~\cite[Proposition 3.1]{AP15}. We only need to state the result for prescribed positive-definite Ricci curvature. Recall that $\zeta_1,\zeta_2\ge0$ are given by~\eqref{zeta_def}. 

\begin{proposition}
\label{prop_Artem}
Suppose $s=2$ in~(\ref{m_decomp}). Consider a metric $T\in\mathcal M$ given by~(\ref{T_def}). Let $G$ have a connected Lie subgroup $K$ satisfying~(\ref{HKGEq}) and~(\ref{km1hEq}). Assume~(\ref{QHypEq}) holds. A metric $g\in\mathcal M$ such that $\Ric g=cT$ for some $c>0$ exists
if and only if
\begin{align}\label{2sum_cond}
\bigg(\zeta_2+\frac{\gamma_{22}^2}{4d_2}+\frac{\gamma_{22}^1}{d_2}\bigg)z_1>\bigg(\zeta_1+\frac{\gamma_{11}^1}{4d_1}\bigg)z_2.
\end{align}
When $g$ exists, it is unique up to scaling.
\end{proposition}

\begin{remark}
{\rm
\label{}
It is common in the literature on compact homogeneous spaces 
to  make the assumption that the fundamental group 
is finite (see, e.g., \cite{WDMK08,Wang2012}).
Theorem~\ref{max_uniqThm} and Proposition~\ref{prop_Artem} imply that all homogeneous spaces considered in this article have finite fundamental groups since they admit  metrics of positive Ricci curvature. 
}
\end{remark}

\section{Two irreducible isotropy summands}\label{sec_2sum}

This section contains the proofs of Theorems~\ref{s2Thm}
(Ricci iteration for $s=2$) and~\ref{AncientThm} (ancient Ricci iterations for $s=2$), as well as two examples. Since it is easy to get lost in the technical lemmas we provide here a quick roadmap to this section: 
\S\ref{polySubS} gives several easy results relating the prescribed curvature equation to zeros of polynomials when $s=2$;
\S\ref{subsec_2max} proves Theorem~\ref{s2Thm}~(i) concerning
the Ricci iteration with maximal isotropy, the main technical tool being Lemma \ref{lem_max} that provides monotonicity of the zeros of the polynomials from \S\ref{polySubS};
\S\ref{subsec_non-max} proves Theorem~\ref{s2Thm}~(ii) concerning
the Ricci iteration with non-maximal isotropy, where here Lemma \ref{lem_nonmax} provides the desired monotonicity;
\S\ref{ancientSubSec} proves Theorem~\ref{AncientThm} concerning
the classification of ancient Ricci iterations, making use of both
Lemmas \ref{lem_max} and  \ref{lem_nonmax};
\S\ref{examSubSec} concludes this section with a collection of examples.

\subsection{Polynomial equations for the prescribed curvature problem}
\lb{polySubS}

We assume the number $s$ in~\eqref{m_decomp} equals~2. Thus, $\mathfrak m$ admits a $Q$-orthogonal $\Ad_G(H)$-invariant decomposition
\begin{align*}
\mathfrak m=\mathfrak m_1\oplus\mathfrak m_2
\end{align*}
such that $\Ad_G(H)|_{\mathfrak m_1}$ and $\Ad_G(H)|_{\mathfrak m_2}$ are irreducible.

As in~\eqref{gEq} and~\eqref{T_def}, we will typically use the notation $x_i$ for the components of a Riemannian metric $g\in\mathcal M$ whose Ricci curvature is prescribed and the notation $z_i$ for the components of the ``Ricci candidate"~$T$. The letter $\alpha$ (possibly with a subscript) will usually stand for the ratio of the two components of a tensor.

We further impose the assumption \eqref{InEquivAssumEq}. The formula for the Ricci curvature in this case simplifies.

\begin{lemma}
\label{PRCs2Lemma}
Suppose that $s=2$ in~(\ref{m_decomp}) and that~(\ref{InEquivAssumEq}) holds.
If $g\in\mathcal M$ is given by~(\ref{gEq}), then $\Ric g$ satisfies~(\ref{RicgEq}) with
\begin{align}
\label{PRC_alg_max}
r_1 &=\frac{b_1}2-\frac{\gamma_{11}^1}{4d_1}-\frac{\gamma_{22}^1}{2d_1}+\frac{\gamma_{22}^1}{4d_1}\alpha^2-\frac{\gamma_{11}^2}{2d_1}\frac1{\alpha}, \notag
\\
r_2 &=\frac{b_2}2-\frac{\gamma_{22}^2}{4d_2}-\frac{\gamma_{11}^2}{2d_2}+\frac{\gamma_{11}^2}{4d_2}\frac1{\alpha^2}-\frac{\gamma_{22}^1}{2d_2}\alpha,
\end{align}
where $\alpha=\frac{x_1}{x_2}$.
\end{lemma}

\begin{proof}
The assumption \eqref{InEquivAssumEq} implies that
all $G$-invariant tensors are completely determined by two real numbers, as in 
\eqref{T_def}. Thus,~\eqref{RicgEq} holds.
Using Lemma~\ref{PRCLemma}, we compute
\begin{align*}
r_1&=
\frac{b_1}{2}
+\frac1{4d_1}\sum_{j,k=1}^2\frac{x_1^2}{x_jx_k}\gamma_{jk}^1
-\frac1{2d_1}\sum_{j,k=1}^2\frac{x_j}{x_k}\gamma_{jk}^1
\\ &=\frac{b_1}{2}
+\frac1{4d_1}\left(\gamma_{11}^1+\frac{x_1^2}{x_2^2}\gamma_{22}^1+
2\frac{x_1}{x_2}\gamma_{12}^1\right)
-\frac1{2d_1}\left(\gamma_{11}^1+\gamma_{22}^1+\frac{x_1}{x_2}\gamma_{12}^1+\frac{x_2}{x_1}\gamma_{21}^1\right)
\\ &=\frac{b_1}2-\frac{\gamma_{11}^1}{4d_1}-\frac{\gamma_{22}^1}{2d_1}+\frac{\gamma_{22}^1}{4d_1}\alpha^2-\frac{\gamma_{11}^2}{2d_1}\frac1{\alpha},
\end{align*}
as desired. The formula for $r_2$ is proved analogously.
\end{proof}

The next result reduces the prescribed Ricci curvature equation to a single polynomial equation. It provides a simple test for whether tensors given by~(\ref{gEq}) and~(\ref{T_def}) form a metric-curvature pair. As we explain in Lemma~\ref{RicPosEinLemma}, the roots of the polynomial $P(x,y)$ appearing in the next result are related to Einstein metrics on~$M$.

\begin{lemma}
\label{PRCPEqLemma}
Suppose that $s=2$ in~(\ref{m_decomp}) and that~(\ref{InEquivAssumEq}) holds.
Let $g\in\mathcal M$ and $T\in \mathcal M$ be given by~(\ref{gEq}) and~(\ref{T_def}). 
The equality $\Ric g=cT$ holds for some $c\ge 0$ if and only if
$P(x_1/x_2,z_1/z_2)=0$, where 
\begin{align}\label{Pxx_def} 
P(x,y):=d_2\gamma_{22}^1x^4+2d_1\gamma_{22}^1yx^3+(\theta_1
-y\theta_2)x^2-2d_2\gamma_{11}^2x-d_1\gamma_{11}^2y
\end{align}
and
\begin{align}\label{th_def} 
\theta_1:=2d_1d_2b_1-d_2\gamma_{11}^1-2d_2\gamma_{22}^1, \qq
\theta_2:=2d_1d_2b_2-d_1\gamma_{22}^2-2d_1\gamma_{11}^2.
\end{align}
\end{lemma}

\begin{proof}
Suppose that $\Ric g=cT$ for some $c\ge0$. Lemma~\ref{PRCs2Lemma} implies
that
\begin{align}
\label{PRC_alg_max2}
\frac{b_1}2-\frac{\gamma_{11}^1}{4d_1}-\frac{\gamma_{22}^1}{2d_1}+\frac{\gamma_{22}^1}{4d_1}\alpha^2-\frac{\gamma_{11}^2}{2d_1}\frac1{\alpha}&=cz_1, \notag
\\
\frac{b_2}2-\frac{\gamma_{22}^2}{4d_2}-\frac{\gamma_{11}^2}{2d_2}+\frac{\gamma_{11}^2}{4d_2}\frac1{\alpha^2}-\frac{\gamma_{22}^1}{2d_2}\alpha&=cz_2,
\end{align}
where $\alpha=\frac{x_1}{x_2}$. Since $T\in\mathcal M$, the numbers $z_1,z_2$ are positive, and so 
\begin{equation}
\begin{aligned}
\label{PrearraEq}
\frac{b_1}2-\frac{\gamma_{11}^1}{4d_1}-\frac{\gamma_{22}^1}{2d_1}+\frac{\gamma_{22}^1}{4d_1}\alpha^2-\frac{\gamma_{11}^2}{2d_1}\frac1{\alpha}
=\frac{z_1}{z_2}\Big(
\frac{b_2}2-\frac{\gamma_{22}^2}{4d_2}-\frac{\gamma_{11}^2}{2d_2}+\frac{\gamma_{11}^2}{4d_2}\frac1{\alpha^2}-\frac{\gamma_{22}^1}{2d_2}\alpha
\Big).
\end{aligned}
\end{equation}
Multiplying by $4d_1d_2\a^2$ and rearranging yields $P(x_1/x_2,z_1/z_2)=0$.

Conversely, suppose that $P(x_1/x_2,z_1/z_2)=0$, and hence \eqref{PrearraEq} holds. Define $c$ as the left-hand side of~\eqref{PrearraEq} divided by $z_1$.
Then, it follows from~\eqref{PrearraEq} that~\eqref{PRC_alg_max2} holds.
By Lemma~\ref{RicPosLemma}, $c\ge0$.
\end{proof}

Denote, for $x>0$,
\begin{equation}
\begin{aligned}
\label{gxEq}
g^x:=x\pi_1^* Q+\pi_2^* Q\in\calM.
\end{aligned}
\end{equation}
The following result characterizes Einstein metrics in $\mathcal M$; cf.~\cite[\S3]{MWWZ86} and~\cite[Section~3]{WDMK08}.

\begin{lemma}
\label{RicPosEinLemma}
Suppose that $s=2$ in~(\ref{m_decomp}) and that~(\ref{InEquivAssumEq}) holds. Then $P(x,x)=0$ if and only if $g^x$ is Einstein with $\Ric g^x=cg^x$ for some $c\ge0$.
\end{lemma}

\begin{proof}
This follows from Lemma~\ref{PRCPEqLemma}.
\end{proof}

We state two identities that follow directly from~\eqref{Casimir}.
In particular, they show that $\th_1,\th_2$ are non-negative.

\begin{lemma}
\label{thetaLemma}
One has (recall  (\ref{th_def})) 
\begin{align*} 
\theta_1=
4d_1d_2\zeta_1+d_2\gamma_{11}^1+4d_2\gamma_{11}^2, \qq
\theta_2=
4d_1d_2\zeta_2+d_1\gamma_{22}^2+4d_1\gamma_{22}^1.
\end{align*}
\end{lemma}

\subsection{Ricci iteration with maximal isotropy}\label{subsec_2max}

Let us prove Theorem~\ref{s2Thm}~(i). We assume $H$ is maximal in~$G$. 
According to~\cite{MWWZ86}, 
in this case, a 
$G$-invariant Einstein metric on $M$ is known to exist. Since $P(x,x)$ factors as $x$ times a cubic polynomial, Lemma~\ref{RicPosEinLemma} implies that there are at most three such metrics, up to scaling; cf.~\cite[\S3]{WDMK08}.

\begin{remark} {\rm
\label{rem_iden_max}
Suppose $g_\infty$ is the limit of the sequence $\{g_i\}_{i\in\mathbb N}$ in Theorem~\ref{s2Thm}~(i). 
When there exist more than one $G$-invariant Einstein metric (up to scaling) on $M$, the proof we are about to present identifies $g_\infty$. More precisely, suppose $\eps_1\le \eps_2\le \eps_3$ are positive numbers such that $\Ric g^{\epsilon_i}= c_ig^{\epsilon_i}$ with $c_i>0$ for all $i\in\{1,2,3\}$.
Assume every $x>0$ satisfying the equation $\Ric g^x= cg^x$ for some $c>0$ coincides with $\epsilon_1$, $\epsilon_2$ or $\epsilon_3$. In the notation of~\eqref{eq_alpha_etc} and~\eqref{alphaiEq}, if 
$\a_{T}\in(0,\eps_1)$, then $g_\infty$ is proportional to $g^{\eps_1}$; if $\a_{T}\in(\eps_3,\infty)$, then $g_\infty$ is proportional to $g^{\eps_3}$; if $\a_{T}\in(\eps_i,\eps_{i+1})$ ($i\in\{1,2\}$), then $g_\infty$ is proportional to $g^{\eps_i}$ or $g^{\eps_{i+1}}$ depending on whether $\a_2<\a_T$
or $\a_2>\a_T$, respectively.
} \end{remark}

\begin{proof}[Proof of Theorem~\ref{s2Thm} (i).]
The existence of the sequence follows from Theorem~\ref{max_uniqThm}, 
while the uniqueness is the content of Lemma~\ref{uniqHMaxLemma} established below. To prove the convergence, suppose
\begin{equation}
\begin{aligned}
\label{giEq}
g_i=x_1^{(i)}\pi_1^*Q+x_2^{(i)}\pi_2^*Q,
\end{aligned}
\end{equation}
and let
\begin{equation}
\begin{aligned}
\label{alphaiEq}
\alpha_i:=x_1^{(i)}/x_2^{(i)}.
\end{aligned}
\end{equation}
Lemma~\ref{lem_max} below demonstrates that $\{\alpha_i\}_{i\in\NN}$ is monotone, bounded below by
$\min\{\alpha_1,\alpha_-\}$
and bounded above by
$\max\{\alpha_1,\alpha_+\}$ (recall~\eqref{alphapmEq}).
Consequently, this sequence has a {\it positive} limit $\alpha_\infty$. 
Thus, if the limits
$$
x_1^{(\infty)}:=\lim_{i\to\infty} x_1^{(i)}
\h{\ \ and \ \ }
x_2^{(\infty)}:=\lim_{i\to\infty} x_2^{(i)}
$$ 
both exist and are finite, then they must be simultaneously positive or zero.
Note that 
$$
\Ric(g_{i+1}/x_2^{(i+1)})=
\Ric g_{i+1}=g_i,
$$
i.e.,
\begin{equation*}
\begin{aligned}
\Ric(\alpha_{i+1}\pi_1^*Q+\pi_2^*Q)=x_1^{(i)}\pi_1^*Q+x_2^{(i)}\pi_2^*Q.
\end{aligned}
\end{equation*}
This shows that the limits indeed both exist and are finite, as the left-hand side has a well-defined limit
by Lemma~\ref{PRCs2Lemma} (as $\a_\infty>0$). We thus have
$$
\Ric (\alpha_\infty\pi_1^*Q+\pi_2^*Q)=x_1^{(\infty)}\pi_1^*Q+x_2^{(\infty)}\pi_2^*Q.
$$
By Lemma~\ref{RicPosLemma} the right-hand side cannot vanish; we conclude that 
$x_1^{(\infty)}, x_2^{(\infty)}>0$.
Finally, let 
\begin{align*}
g_\infty=x_1^{(\infty)}\pi_1^*Q+x_2^{(\infty)}\pi_2^*Q.
\end{align*}
We have shown that $\{g_i\}_{i\in\mathbb N}$ converges to $g_\infty$, and as explained in~\S\ref{prelim_subsec}, this convergence is smooth.
By passing to the limit in formula~\eqref{RIEq}, we see that $\Ric g_\infty=g_\infty$, concluding the proof.
\end{proof}

\begin{lemma}
\label{uniqHMaxLemma}
Suppose that $s=2$ in~(\ref{m_decomp}) and that~(\ref{InEquivAssumEq}) holds. Choose $T\in\mathcal M$. Up to scaling, there exists a unique metric $g\in\mathcal M$ whose Ricci curvature equals $cT$ for some $c>0$.
\end{lemma}

\begin{proof}
The existence of $g$ follows from Theorem~\ref{max_uniqThm}. To prove the uniqueness, fix $b>0$ and set
\begin{equation*}
\begin{aligned}
 P_b(x):=P(x,b). 
\end{aligned}
\end{equation*}
By Lemma \ref{PRCPEqLemma}, it suffices to show that $P_b(x)=0$ for at most one $x\in(0,\infty)$.

The
leading coefficient of $P_b$ (i.e., the coefficient of $x^4$) is positive
by Lemma \ref{gammaMaxLemma}, so the leading coefficient of the quadratic polynomial 
$R_2:=\frac{d^2}{dx^2}P_b(x)$ is positive. The 
  $x$ coefficient of $R_2$ is nonnegative. Thus, $R_2$ has no more than one positive root, equivalently, $R_1:=\frac d{dx}P_b(x)$ has at most one critical point, in $(0,\infty)$. Since $R_1 (0) < 0$ and $ \lim_{x\ra\infty}R_1 (x) =\infty$ it follows that
$R_1$ has at most one positive root. 
Consequently, the function $P_b$ has at most one critical point in $(0,\infty)$. 
Again, since $P_b(0) \le 0$ and $ \lim_{x\ra\infty}P_b (x) =\infty$, it follows that $P_b$ has no more than one positive root.
\end{proof}

\begin{remark} {\rm
 In regards to the previous lemma, we remark that in the context of the Ricci iteration we have a naturally imposed scaling normalization given by~\eqref{RIEq}.
Namely, assume that the hypotheses of Theorem~\ref{s2Thm}~(i) hold. Let $g$ be a metric such that $\Ric g=g_{i-1}$. Then,
$\Ric cg=g_{i-1}$ for any $c>0$. However, according to Lemma~\ref{uniqHMaxLemma}, there is at most one 
$c>0$ such that 
$\Ric g_{i+1}=cg$ is solvable for some $g_{i+1}$, and we set $g_i=cg$.}
\end{remark}

Let $T,g,h\in\mathcal M$ be given by~\eqref{T_def}, \eqref{gEq} and
\begin{align}\label{gh_def}
h:=
y_1\pi_1^*Q+y_2\pi_2^*Q.
\end{align}
Denote 
\begin{align}\label{eq_alpha_etc}
\alpha_T:=\frac{z_1}{z_2},
\q
\alpha_g:=\frac{x_1}{x_2},
\q
\alpha_h:=\frac{y_1}{y_2}.
\end{align}
The next result gives the monotonicity needed in the proof of Theorem~\ref{s2Thm}. Before stating it, we need to point out that, as Lemma~\ref{RicPosEinLemma} shows, there are at most three numbers $x>0$ such that $g^x\in\mathcal E$ (recall~\eqref{EEq} and \eqref{gxEq}).

\begin{lemma}\label{lem_max}
Suppose that $H$ is maximal in $G$, that $s = 2 $ in~(\ref{m_decomp}), and that~(\ref{InEquivAssumEq}) holds.
Let the metrics $T,g,h\in\mathcal M$ given by~(\ref{T_def}), (\ref{gEq}) and~(\ref{gh_def})
satisfy $\Ric g=cT$ for some $c>0$ and $\Ric h=g$. Suppose $\eps,\epsilon_1,\epsilon_2>0$ are such that $g^\eps,g^{\eps_1},g^{\eps_2}\in\mathcal E$ and $g^x\not\in \mathcal E$ whenever
$x\in(\eps_1,\eps_2)$. Then:
\begin{enumerate}[(i)]
\item
If $\alpha_T\ge\eps$, then
$\alpha_g\ge \eps$.
\item
If $\alpha_T\le\eps$, then
$\alpha_g\le \eps$.
\item
If $\alpha_T\ge\a_+$ (recall (\ref{alphapmEq})), then
$\alpha_T\ge\alpha_g\ge \a_+$.
\item
If $\eps_1\le\alpha_g\le\alpha_T\le \eps_2$, then
$\eps_1\le\alpha_h\le\alpha_g$.
\item
If 
$\eps_1\le\alpha_T\le\alpha_g\le\eps_2$, then
$\alpha_g\le\alpha_h\le\eps_2$.
\item
If $\alpha_T\le\a_-$ (recall~(\ref{alphapmEq})), then
$\alpha_T\le\alpha_g\le\a_-$.
\end{enumerate}
\end{lemma}

\begin{proof}
(i) 
The function $P(\eps,y)$ is decreasing in~$y$. Indeed,
since 
$\Ric g^\eps = cg^\eps$ for some $c>0$, the second line in~\eqref{PRC_alg_max} implies
$$
\baeq
c&=
\frac{b_2}2-\frac{\gamma_{22}^2}{4d_2}-\frac{\gamma_{11}^2}{2d_2}+\frac{\gamma_{11}^2}{4d_2}
\frac1{\eps^2}-\frac{\gamma_{22}^1}{2d_2}\eps
\cr
&=
(\th_2
+d_1\gamma_{11}^2\frac1{\eps^2}
-2d_1\gamma_{22}^1\eps
)
/4d_1d_2
.
\eaeq
$$
Thus,
\begin{align*}
\frac\partial{\partial y}P(\eps,y)=2d_1\gamma_{22}^1\eps^3
-\theta_2\eps^2-d_1\gamma_{11}^2=-4d_1d_2\eps^2c<0,
\end{align*}
as claimed.
 Note that $P(\eps,\eps)=0$ by Lemma \ref{RicPosEinLemma}. If $\alpha_T\ge\eps$, then
\begin{align}
\lb{PepsTEq}
P(\eps,\alpha_T)\le P(\eps,\eps)=0.
\end{align}
As in the proof of Lemma~\ref{uniqHMaxLemma}, the leading coefficient of the polynomial $P_{\alpha_T}(x)=P(x,\alpha_T)$ 
is positive, so $\lim_{x\ra\infty}P_{\alpha_T}(x)=\infty$.  
From this and~\eqref{PepsTEq}, it follows that there exists $x\ge\eps$ such that 
$P_{\alpha_T}(x)=0$. However, 
Lemma \ref{PRCPEqLemma} implies that
$P(\a_g,\alpha_T)=0$, and Lemma \ref{uniqHMaxLemma} then implies that
 $x=\alpha_g$, so $\alpha_g\ge\eps $ as desired.

(ii) The argument of (i) gives that $P(\eps,\alpha_T)\ge 0$. Note that
$P(0,\alpha_T)=
-d_1\gamma_{11}^2\alpha_T<0$ by  Lemma \ref{gammaMaxLemma}. 
Thus, $P(x,\a_T)=0$ for some $x\in(0,\a_T]$, and the rest of the proof follows that of
(i).

(iii) Suppose $\alpha_T\ge\a_+$. Setting $\eps=\a_+$ in (i), we obtain
\begin{equation}
\begin{aligned}
\label{alphagaplusEq}
 \alpha_g\ge\a_+.
\end{aligned}
\end{equation}
It remains to show that  $\alpha_g\le\alpha_T$. 
Formul\ae\ \eqref{PRC_alg_max} yield
\begin{align*}
\alpha_g-\alpha_T&=\alpha_g-\frac{\theta_1\alpha_g^2+d_2\gamma_{22}^1\alpha_g^4-2d_2\gamma_{11}^2\alpha_g}{
\theta_2\alpha_g^2+d_1\gamma_{11}^2-2d_1\gamma_{22}^1\alpha_g^3} \notag \\
&=\frac{\alpha_g\big(-(2d_1+d_2)\gamma_{22}^1\alpha_g^3+\theta_2\alpha_g^2-\theta_1\alpha_g+(d_1+2d_2)\gamma_{11}^2\big)}{
\theta_2\alpha_g^2+d_1\gamma_{11}^2-2d_1\gamma_{22}^1\alpha_g^3}\,.
\end{align*}
Since $\Ric g=cT$, the denominator is a positive multiple of the second component of $T$, i.e., is itself positive. Thus, the sign of $\alpha_g-\alpha_T$ coincides with the sign of $\tilde P(\alpha_g)$, where
\begin{align*}
\tilde P(x):=-(2d_1+d_2)\gamma_{22}^1x^3+\theta_2x^2-\theta_1x+(d_1+2d_2)\gamma_{11}^2,\qquad x\in\mathbb R.
\end{align*}
Note that $\tilde P(x)$ is a polynomial whose leading coefficient is
negative. To prove that $\alpha_g\le\alpha_T$, it suffices to show that $\alpha_g$ is greater than or equal to the largest root of this polynomial.
Since
\begin{align*}
\tilde P(x)=-{P(x,x)}/x
\end{align*}
for all $x>0$, the positive roots of $\tilde P(x)$ are precisely all $\eps$ such that $g^\eps\in\calE$, and we are done by~\eqref{alphagaplusEq}.

(iv) Assume $\eps_1\le\alpha_g\le\alpha_T\le\eps_2$. 
By assumption,
\begin{equation}
\begin{aligned}
\label{pxxnotzeroEq}
P(x,x)\not=0,\qquad x\in(\eps_1,\eps_2).
\end{aligned}
\end{equation}
By (i) and (ii), $\eps_1\le\a_h\le\eps_2$. It suffices to show that $\a_h-\alpha_g\le0$ or, equivalently, that $\tilde P(\a_h)\le0$. 
Since $\tilde P(\eps_1)=\tilde P(\eps_2)=0$, assume that $\a_h\in(\eps_1,\eps_2)$.
We may also assume $\eps_1<\alpha_g<\alpha_T<\eps_2$
(otherwise, Lemma \ref{RicPosEinLemma} implies $g$ and $h$ are Einstein, which means $\a_g=\a_h$, and we are done). 
By~\eqref{pxxnotzeroEq}, $\tilde P(x)$ does not change sign
in the interval $(\eps_1,\eps_2)$. Therefore, it suffices 
to prove that $\tilde P(\alpha_g)<0$. But the sign of $\tilde P(\alpha_g)$ coincides with the sign of $\alpha_g-\alpha_T$. Since $\alpha_g<\alpha_T$ by assumption, we are done.

(v) The proof is essentially identical to that of~(iv).

(vi) The proof is essentially identical to that of~(iii).
\end{proof}

\subsection{Ricci iteration with non-maximal isotropy}\label{subsec_non-max}

Let us prove Theorem~\ref{s2Thm}~(ii). We assume $G$ has a connected Lie subgroup $K$ such that~\eqref{HKGEq} holds. The number $s$ in~\eqref{m_decomp} is still~2, and~\eqref{km1hEq} is satisfied. 
Since $\mathfrak k$ is a Lie algebra, $[\mathfrak k,\mathfrak k]\subset
\mathfrak k$, and as $\mathfrak k$
 is $Q$-orthogonal to $\mathfrak m_2$,
\begin{equation}
\begin{aligned}
\label{gamma112Eq}
\gamma_{11}^2=0.
\end{aligned}
\end{equation}
The requirement~\eqref{QHypEq} is equivalent to the formula
\begin{equation}
\begin{aligned}
\label{gamma211Eq}
\gamma_{22}^1>0. 
\end{aligned}
\end{equation}
If $\gamma_{22}^1=0$, then Lemma~\ref{PRCs2Lemma} implies that
all the metrics in $\mathcal M$ have the same Ricci curvature, namely,
\begin{equation}
\begin{aligned}
\label{SameRicEq}
\Big(
\frac{b_1}2-\frac{\gamma_{11}^1}{4d_1}
\Big)
\pi_1^* Q
+
\Big(
\frac{b_2}2-\frac{\gamma_{22}^2}{4d_2}
\Big)
\pi_2^* Q.
\end{aligned}
\end{equation}
In this case, the Ricci iteration exists if and only if the initial metric is equal to~\eqref{SameRicEq}, and hence is Einstein.

Because~\eqref{gamma112Eq} holds, the expression $P(x,x)$ is now a quadratic polynomial in $x$ times $x^2$. Hence Lemma~\ref{RicPosEinLemma} implies that there are at most two Einstein metrics in $\calE$, up to scaling; cf.~\cite[\S3]{WDMK08}.

\begin{remark} {\rm
\label{}
By analogy with Remark~\ref{rem_iden_max}, our proof Theorem~\ref{s2Thm}~(ii) identifies the limit, $g_\infty$, of the sequence $\{g_i\}_{i\in\mathbb N}$, provided this sequence exists. More precisely, suppose there are two distinct Einstein
metrics of volume 1 in~$\calM$. In the notation of~\eqref{alphapmEq} and~\eqref{eq_alpha_etc}, if 
$\a_{T}\in(0,\a_-)$, then $g_\infty$ is proportional to $g^{\a_-}$; if $\a_{T}\in(\a_-,\infty)$, then $g_\infty$ is proportional to $g^{\a_+}$.
} \end{remark}

\begin{proof}[Proof of Theorem \ref{s2Thm} (ii).]
(b) Suppose $\Ad_G(H)|_{\mathfrak m_1}$ is nontrivial, 
$\mathcal E\ne\emptyset$ and $\alpha_T\ge\a_-$.
Then by Remark \ref{zetaRemark} one has
\begin{equation}
\begin{aligned}
\label{zeta1Eq}
\zeta_1>0.
\end{aligned}
\end{equation}  
Lemma~\ref{lem_nonmax}~(i) below implies that a sequence 
$\{g_i\}_{i\in\mathbb N}$ satisfying~\eqref{RIEq} exists and is unique.
Let $\a_i$ be as in~\eqref{giEq} and~\eqref{alphaiEq}.
Lemma~\ref{lem_nonmax}~(i) also implies
that
the sequence $\{\alpha_i\}_{i\in\NN}$ is monotone and
\begin{align*}
\a_-\le \alpha_i\le\max\{\alpha_1,\alpha_+\},\qquad i\in\mathbb N.
\end{align*}
Therefore, this sequence must converge to some
number~$\alpha_\infty$. It is clear that $\alpha_\infty\ge\alpha_->0$. An alternative way to see that 
$\alpha_\infty>0$ is to observe that, by Proposition~\ref{prop_Artem}, 
\begin{equation}
\begin{aligned}
\label{aietaEq}
\alpha_i>
\frac{\zeta_1+\frac{\gamma_{11}^1}{4d_1}}{\zeta_2+\frac{\gamma_{22}^2}{4d_2}+\frac{\gamma_{22}^1}{d_2}}
=\frac{\tilde\eta_1}{\tilde\eta_2},
\q i\in\mathbb N,
\end{aligned}
\end{equation} 
where $\tilde\eta_1$ and $\tilde\eta_2$ denote the numerator and the denominator of the middle expression. The positivity of $\alpha_\infty$ then follows from~\eqref{gamma211Eq} and~\eqref{zeta1Eq}.

As in the proof of Theorem~\ref{s2Thm}~(i), the Ricci iteration equation~\eqref{RIEq}
implies that 
$$
x_1^{(\infty)}=\lim_{i\to\infty} x_1^{(i)}
\h{\ \ and \ \ }
x_2^{(\infty)}=\lim_{i\to\infty} x_2^{(i)}
$$ 
exist and are finite. Since $\alpha_\infty>0$, these limits
are simultaneously positive or zero. Plugging~\eqref{gamma112Eq} in Lemma~\ref{PRCs2Lemma} gives
\begin{align}\label{eq_Ricci_algHnonmax}
x_1^{(i)}&:=
\frac{b_1}2-\frac{\gamma_{11}^1}{4d_1}-\frac{\gamma_{22}^1}{2d_1}+\frac{\gamma_{22}^1}{4d_1}\alpha_{i+1}^2, \notag
\\
x_2^{(i)}&:=
\frac{b_2}2-\frac{\gamma_{22}^2}{4d_2}-\frac{\gamma_{22}^1}{2d_2}\alpha_{i+1}.
\end{align}
Combining~\eqref{eq_Ricci_algHnonmax}, \eqref{Casimir} and~\eqref{zeta1Eq} yields
\begin{align*}
x_1^{(\infty)}&=
\zeta_1+\frac{\gamma_{11}^1}{4d_1}+\frac{\gamma_{22}^1}{4d_1}\alpha_\infty^2
>0, 
\end{align*}
and so $x_2^{(\infty)}$ is positive.
As in the proof of Theorem~\ref{s2Thm}~(i), it follows that the Ricci iteration
converges smoothly to some $g_\infty\in\mathcal M$. Passing to the limit shows that $g_\infty\in\mathcal E$.

Assume now that $\a_T<\a_-$. Suppose the sequence $\{g_i\}_{i\in\NN}$ exists. 
Then it must be unique by Proposition~\ref{prop_Artem}.
The corresponding sequence $\{\alpha_i\}_{i\in\NN}$ is monotone non-increasing by Lemma~\ref{lem_nonmax}~(i) and converges to some $\alpha_\infty$.
Because~\eqref{aietaEq} holds, the reasoning above shows that
$\alpha_\infty\in[{\tilde\eta_1}/{\tilde\eta_2},\a_-)\subset (0,\a_-)$ and $\{g_i\}_{i\in\mathbb N}$ converges to an Einstein metric $g_\infty$ proportional to $g^{\alpha_\infty}$ (recall~\eqref{gxEq}). But
then $g^{\alpha_\infty}$ is Einstein, which is impossible since $\a_\infty<\a_-$.

(a) Suppose $\Ad_G(H)|_{\mathfrak m_1}$ is trivial. Lemma~\ref{lem_nonmax}~(ii) implies that a sequence $\{g_i\}_{i\in\NN}$ satisfying~\eqref{RIEq} exists and is unique.
As before, let $\a_i$ be as in~\eqref{giEq} and~\eqref{alphaiEq}.
Lemma~\ref{lem_nonmax}~(ii) also implies
that the sequence $\{\alpha_i\}_{i\in\NN}$ is monotone and
\begin{align*}
\min\{\alpha_1,\alpha_-\}=\min\{\alpha_1,\alpha_+\}\le\alpha_i\le\max\{\alpha_1,\alpha_+\}=\max\{\alpha_1,\alpha_-\},\qquad i\in\mathbb N.
\end{align*}
Therefore, $\{\alpha_i\}_{i\in\NN}$ must converge to some positive 
number~$\alpha_\infty$. As in the proof of~(b), we use this fact to show that $\{g_i\}_{i\in\mathbb N}$ converges to a $G$-invariant Einstein metric on $M$. The uniqueness of such metrics up to scaling follows from the equality $\alpha_-=\alpha_+$.

(c) Suppose $\mathcal E=\emptyset$ and a sequence $\{g_i\}_{i\in\mathbb N}$ satisfying~\eqref{RIEq} exists. Lemma~\ref{lem_nonmax}~(iii) shows that $\{\alpha_i\}_{i\in\NN}$ is monotone decreasing. Arguing as in the proof of the ``if" portion of~(b), we arrive at a contradiction.
\end{proof}

\begin{lemma}\label{lem_nonmax}
Suppose that $s=2$ in~(\ref{m_decomp}) and that~(\ref{InEquivAssumEq}) holds. Consider a metric $T\in\mathcal M$ given by~(\ref{T_def}). Let $G$ have a connected Lie subgroup $K$ satisfying~(\ref{HKGEq}). Assume~(\ref{km1hEq}) and~(\ref{QHypEq}) hold.
\begin{enumerate}[(i)]
\item
Suppose $\Ad_G(H)|_{\mathfrak m_1}$ is nontrivial and $\calE\ne\emptyset$. Then:
\begin{enumerate}
\item
If $\alpha_T\ge\a_+$, then
there exists a metric $g\in\mathcal M$, unique up to scaling, such that $\Ric g=cT$ for some $c>0$.
Moreover, $\alpha_T\ge\alpha_g\ge \a_+$.

\item
If $\alpha_T\in(\a_-,\a_+)$, then
there exists $g\in\mathcal M$, unique up to scaling, such that $\Ric g=cT$ for some $c>0$.
Moreover, $\alpha_T\le\alpha_g\le \a_+$.

\item
If $\alpha_T\le\a_-$ and $\Ric g=cT$ for some $g\in\mathcal M$ and $c>0$, 
then $\alpha_g\le\alpha_T$.
\end{enumerate}

\item
Suppose $\Ad_G(H)|_{\mathfrak m_1}$ is trivial. Then:
\begin{enumerate}
\item
The set $\calE$ is nonempty, and the equality $\alpha_-=\alpha_+$ holds true.

\item
There exists a metric $g\in\mathcal M$, unique up to scaling, such that $\Ric g=cT$ for some $c>0$.

\item
If $\alpha_T\ge\alpha_-=\alpha_+$, then $\alpha_T\ge\alpha_g\ge\alpha_-=\alpha_+$.

\item
If $\alpha_T\le\alpha_-=\alpha_+$,
then $\alpha_T\le\alpha_g\le\alpha_-=\alpha_+$.
\end{enumerate}

\item
If $\calE=\emptyset$ and $\Ric g=T$ for some $g\in\mathcal M$, then $\alpha_T\ge\alpha_g$.

\end{enumerate}
\end{lemma}

\begin{remark} {\rm
\label{}
Lemma~\ref{lem_nonmax} can be proven in essentially the same way as Lemma~\ref{lem_max}. We provide an alternative proof below, which some readers may find conceptually simpler.
} \end{remark}

\begin{proof}
It will be convenient to use the slightly different normalization
than  \eqref{aietaEq}:
\begin{align*}
\eta_i:=\frac{2d_2^2}{d_1\gamma_{22}^1}\tilde\eta_i,
\q i\in\{1,2\}.
\end{align*}
Observe that $\eta_2\not=0$. Condition~\eqref{2sum_cond} means 
$z_1/z_2>\eta_1/\eta_2$.

(i-a) The metric $g^{\a_+}$ is Einstein. This fact and
Proposition~\ref{prop_Artem} imply $\alpha_+>\frac{\eta_1}{\eta_2}$, 
so our assumption gives $\a_T> \frac{\eta_1}{\eta_2}$.
One more application of Proposition~\ref{prop_Artem} yields the existence and the uniqueness of $g\in\mathcal M$ such that $\Ric g=cT$.

Next, we show that $\alpha_+\le\alpha_g$. 
Lemmas~\ref{PRCs2Lemma} and~\ref{thetaLemma}, together with~\eqref{gamma112Eq}, give
$$
\a_T=\frac{\tilde\eta_1+\gamma^1_{22}\a_g^2/4d_1}{\tilde\eta_2-\gamma^1_{22}\a_g/2d_2}
=
\frac{\eta_1+\frac12(d_2\a_g/d_1)^2}{\eta_2-d_2\a_g/d_1},
$$
so letting $y=d_2\a_g/d_1$, we obtain
$$
y^2+2\a_Ty+2\eta_1-2\eta_2\a_T=0.
$$ 
Consequently,
$\alpha_g=F(\a_T)$ with
\begin{align*}
F(x)=\frac{d_1}{d_2}
\big(
\sqrt{x^2+2\eta_2x-2\eta_1}-x
\big),\quad x\in({\eta_1}/{\eta_2},\infty).
\end{align*}
Because $g^{\a_-}$ and $g^{\a_+}$ are Einstein, we have
$\alpha_-=F(\alpha_-)$ and $\alpha_+=F(\alpha_+)$. This implies
\begin{align*}
\alpha_g-\alpha_+=F(\a_T)-F(\alpha_+).
\end{align*}
Since $\a_T\ge\alpha_+$, it suffices to demonstrate that $F$ is an increasing function, to wit,
\begin{align*}
F'(x)=\frac{d_1}{d_2}
\bigg(
\frac{x+\eta_2-\sqrt{(x+\eta_2)^2-\eta_2^2-2\eta_1}}{\sqrt{x^2+2\eta_2x-2\eta_1}}
\bigg)
>0
,\quad x\in({\eta_1}/{\eta_2},\infty).
\end{align*}

Finally, let us prove that $\alpha_g\le\a_T$, i.e., $F(\a_T)\le \a_T$, or
\begin{align*}
\frac{d_1}{d_2}\sqrt{\a_T^2+2\eta_2\a_T-2\eta_1}
\le\bigg(1+\frac{d_1}{d_2}\bigg)\a_T.
\end{align*}
Squaring the nonnegative expressions on both sides shows that $\a_g\le\a_T$ if and only if $\tilde F(\a_T)\ge0$ with
\begin{align*}
\tilde F(x)=\bigg(1+2\frac{d_1}{d_2}\bigg)x^2-2\frac{d_1^2}{d_2^2}\eta_2x+2\frac{d_1^2}{d_2^2}\eta_1,\qquad x\in\mathbb R.
\end{align*}
When $x>\eta_1/\eta_2$, the formula $\tilde F(x)=0$ implies 
$F(x)=x$. For such $x$, the equality $F(x)=x$ holds if and only if $x=\alpha_-$ or $x=\a_+$; cf.~Lemma~\ref{RicPosEinLemma}. To prove that $\tilde F(\a_T)\ge0$, it suffices to show that $\a_T$ is greater than or equal to the largest root of the polynomial $\tilde F(x)$, i.e., $\a_T\ge\alpha_+$. But this inequality holds by assumption.

(i-b)--(i-c) It suffices to repeat the arguments from the proof of~(i-a) with very minor changes.

(ii-a) Suppose $\Ad_G(H)|_{\mathfrak m_1}$ is trivial. Then $\zeta_1=0$. Because $\mathfrak m_1$ is irreducible, it must be 1-dimensional. Ergo, the constants $\gamma_{11}^1$ and $\eta_1$ equal~0. Lemma~\ref{thetaLemma} shows that $\theta_1$ equals~0 as well. The expression $P(x,x)$ (recall~\eqref{Pxx_def}) factors into $x^3$ and a linear function of $x$ with negative free term. Therefore, it vanishes for exactly one positive $x$. Since $P(\a_-,\a_-)=P(\a_+,\a_+)=0$ by Lemma~\ref{RicPosEinLemma}, we conclude that $\alpha_-=\alpha_+$.

(ii-b) This claim is a direct consequence of Proposition~\ref{prop_Artem} and the equality $\eta_1=0$.

(ii-c)--(ii-d)--(iii) It suffices to repeat the arguments from the proof of~(i-a) with minor modifications.
\end{proof}

\subsection{Ancient Ricci iterations}
\label{ancientSubSec}

In this subsection, we prove Theorem~\ref{AncientThm}. First, we observe that a metric has positive Ricci curvature as soon as it is ``sandwiched" between two metrics with positive Ricci curvature.

\begin{lemma}
\label{RicSandwichLemma}
Suppose that $s=2$ in~(\ref{m_decomp}) and that~(\ref{InEquivAssumEq}) holds.
Let $g_i\in\mathcal M$ and $\a_i>0$ be given by~(\ref{giEq}) 
and~(\ref{alphaiEq}) for $i=1,2,3$. Assume that 
$\a_1\le\a_2\le\a_3$. If $\Ric g_1$ and $\Ric g_3$ are positive-definite, then so is $\Ric g_2$.
\end{lemma}

\begin{proof}
According to Lemma \ref{PRCs2Lemma},
$$\Ric g_i = r^{(i)}_1\pi_1^*Q+r^{(i)}_2\pi_2^*Q$$ with
\begin{align}
\label{PRC_alg_max3}
r^{(i)}_1 &=A_1+B_1\alpha_i^2-\frac {C_1}{\alpha_i}, \notag
\\
r^{(i)}_2 &=A_2+\frac{B_2}{\alpha_i^2}-{C_2}\alpha_i,
\end{align}
where  $A_1,A_2,B_1,B_2,C_1,C_2$ are constants independent of $\a_i$, and where $B_1,B_2,C_1,C_2\ge0$ in view of~\eqref{gamma_def}. By assumption, we have $r^{(1)}_1>0$. 
Since $\a_2\ge\a_1$, 
it follows that $r^{(2)}_1\ge r^{(1)}_1>0$.
Similarly, by assumption, $r^{(3)}_2>0$.
Since $\a_2\le\a_3$, 
it follows that $r^{(2)}_2\ge r^{(3)}_2>0$. Thus,
$\Ric g_2$ is positive-definite.
\end{proof}

In the non-maximal setting we will  use also the following result to guarantee positivity
of the Ricci tensor.

\begin{lemma}
\label{TrivCaseLemma}
Assume that $s = 2$ in~(\ref{m_decomp}) and that~(\ref{InEquivAssumEq}) holds. Suppose $G$ has a connected Lie subgroup $K$ satisfying~(\ref{HKGEq}). Let~(\ref{km1hEq}) and~(\ref{QHypEq}) hold.
If $T\in\mathcal M$ satisfies~(\ref{T_def}) with $z_1/z_2\le \a_-$, then $\Ric T\in\calM$.
\end{lemma}

\begin{proof}
When $z_1/z_2=\a_-$, the result is obvious. Suppose that $z_1/z_2< \a_-$.
Since $\gamma_{11}^2=0$ by \eqref{gamma112Eq},
Lemma~\ref{PRCs2Lemma} gives that
$
\Ric T = r_1\pi_1^*Q+r_2\pi_2^*Q 
$
with
\begin{align*}
r_1 &=A_1+B_1(z_1/z_2)^2,
\\
r_2 &=A_2-{C_2}z_1/z_2,
\end{align*}
where  $A_1,A_2,B_1,C_2$ are constants independent of $T$. Moreover, $B_1,C_2>0$ by~\eqref{gamma211Eq}, and 
\beq\label{A1Eq}
A_1= \th_1/4d_1d_2
\eeq
(recall~\eqref{th_def}).
By Lemma \ref{thetaLemma}, $\th_1=4d_1d_2\zeta_1+d_2\gamma_{11}^1+4d_2\gamma_{11}^2\ge0$,
so $A_1\ge0$.
Thus, $r_1\ge B_1(z_1/z_2)^2>0$, while $r_2>A_2-{C_2}\a_->0$
since $\Ric g^{\a_-}= g^{\a_-}$ is positive-definite. Thus, $\Ric T\in\calM$.
\end{proof}

As in the formulation of Theorem~\ref{AncientThm}, we always assume $T\in\calM$ is given by~\eqref{T_def} in the proof below.

\begin{proof}[Proof of Theorem~\ref{AncientThm}]
(i)
Suppose that the assumptions
of Theorem~\ref{AncientThm}~(i) hold. 
Let $T$ satisfy $\a_-< z_1/z_2<\a_+$. 
Applying Lemma~\ref{RicSandwichLemma}
to the triple $g^{\a_-},T,g^{\a_+}$ shows that 
$\Ric T\in\mathcal M$ since $\Ric g^{\a_-}=
g^{\a_-}$ and $\Ric g^{\a_+}=
g^{\a_+}$ are positive-definite. 
Thus,
$$
\{T\in\calM\,:\, z_1/z_2\in[\a_-,\a_+]\}\subset
\calM^{(2)}.
$$
Next, let $T$ satisfy $\a_-<z_1/z_2<\a_+$,
and set $g_0:=\Ric T\in\mathcal M$. According to Lemma~\ref{lem_max}~(ii),
if we had $\a_0\le\a_-$ (recall~\eqref{alphaiEq}), we would also have $\a_T\le\a_-$.
It follows that $\a_0>\a_-$.
Similarly, according to
Lemma~\ref{lem_max}~(i),
if we had $\a_0\ge\a_+$, then we would have $\a_g\ge\a_+$.
It follows that $\a_0<\a_+$. 
Thus, we may apply Lemma~\ref{RicSandwichLemma}
to the triple $g^{\a_-},g_0,g^{\a_+}$ to conclude that $g_{-1}:=\Ric g_0\in\calM$, i.e.,
$$
\{T\in\calM\,:\, z_1/z_2\in[\a_-,\a_+]\}\subset
\calM^{(3)}.
$$
By induction, it follows that
$$
\{T\in\calM\,:\, z_1/z_2\in[\a_-,\a_+]\}\subset
\calM^{(\infty)}.
$$
Moreover, when $g_1\in\calM$ is such that $\alpha_1\in(\a_-,\a_+)$, the ancient Ricci iteration~\eqref{RRIEq} 
exists and
$\a_{-i}\in
(\a_-,\a_+)$ for all $i\in\NN\cup\{0\}$.
Lemma \ref{lem_max} (iv)--(v) implies that the sequence $\{\a_{-i}\}_{i=-1}^\infty$ is monotone. The arguments in the proof of Theorem \ref{s2Thm} (i) now apply verbatim to show that the limit of $\{g_{-i}\}_{i=-1}^\infty$ exists in the smooth topology and is an Einstein metric.

Finally, suppose that $T\in \calM^{(\infty)}$ and
$z_1/z_2>\a_+$. Then \eqref{RRIEq} with $g_1=T$ is well-defined. The arguments of the previous paragraph show that 
$\lim_{i\ra\infty}\a_{-i}=\infty$, or else $\a_{-i}$ must converge to some $\a_{-\infty}>\a_+$ such that $g^{\a_{-\infty}}$ an Einstein metric,
a contradiction. However, there exists $C>0$ depending only on
$d_1,d_2,\{\gamma^{l}_{jk}\}_{j,k,l=1}^2,b_1,b_2$ such that
$r^{(i)}_2<0$ if $\a_{i}>C$, by Lemma  \ref{PRCs2Lemma} (observing that  
$\gamma_{22}^1>0$ by Lemma  \ref{gammaMaxLemma}).
Thus, $r(T)<\infty$.
Similar reasoning shows that 
$z_1/z_2<\a_-$ implies $r(T)<\infty$
(otherwise $r^{(i)}_1<0$ for some $i$). Thus, we obtain~\eqref{Minfmax}.

(ii) Suppose that the hypoteses of Theorem~\ref{AncientThm}~(ii) hold. First, assume $\Ad_G(H)|_{\mathfrak m_1}$ is nontrivial
and $\mathcal E\ne\emptyset$.
Suppose $r(g_1)\ge 2$, i.e.,
$g_0:=\Ric g_1\in\calM$. 
As above, Lemma \ref{lem_max} (i)--(ii) implies 
that $\a_0\in(\a_-,\a_+)$ when $\a_1\in(\a_-,\a_+)$. Similarly, $\a_0<\a_-$ when $\a_1<\a_-$. 

If $\a_1<\a_-$, then 
Lemma~\ref{lem_nonmax}~(i-c) yields $\a_1<\a_0$.
Applying Lemma~\ref{RicSandwichLemma}
to the triple $g_1,g_0,g^{\a_-}$ shows that 
$g_{-1}:=\Ric g_0\in\mathcal M$ since, by assumption, $\Ric g^{\a_-}=
g^{\a_-}$ and $\Ric g_1$ are positive-definite.
Thus,
$$
\{T\in\calM^{(2)}\,:\, z_1/z_2<\a_-\}\subset
\calM^{(3)}.
$$
By induction, it follows that 
$$
\{T\in\calM^{(2)}\,:\, z_1/z_2<\a_-\}\subset
\calM^{(\infty)}.
$$
In addition, the ancient Ricci iteration \eqref{RRIEq} 
exists, and
$\a_{-i}\in (0,\a_-)$ for all $i\in\NN\cup\{0\}$. We also have
$\a_{-i}<\a_{-i-1}$.

If  $\a_1\in(\a_-,\a_+)$ holds, then applying Lemma~\ref{RicSandwichLemma}
to the triple $g^{\a_-},g_1,g^{\a_+}$ shows that 
$g_{0}:=\Ric g_1\in\mathcal M$ since $\Ric g^{\a_-}=
g^{\a_-}$ and $\Ric g^{\a_-}=
g^{\a_-}$ are positive-definite, so
$$
\{T\in\calM\,:\, \a_-\le z_1/z_2\le \a_+\}\subset
\calM^{(2)}.
$$
(Thus, there is no need to
make the assumption $\Ric g_1\in\mathcal M$.)
Lemma~\ref{lem_nonmax}~(i-b) 
implies that $\a_1>\a_0$. As noted earlier, $\a_0\in(\a_-,\a_+)$. Therefore, applying Lemma~\ref{RicSandwichLemma}
to the triple $g^{\a_-},g_0,g^{\a_+}$ shows that 
$g_{-1}:=\Ric g_0\in\mathcal M$, i.e.,
$$
\{T\in\calM\,:\, \a_-\le z_1/z_2\le \a_+\}\subset
\calM^{(3)}.
$$ 
By induction, it follows that 
$$
\{T\in\calM\,:\, \a_-\le z_1/z_2\le \a_+\}\subset
\calM^{(\infty)}.
$$ 
Also, the ancient Ricci iteration~\eqref{RRIEq} 
exists and
$\a_{-i}\in (\a_-,\a_+)$ for all $i\in\NN\cup\{0\}$. Also,
$\a_{-i}>\a_{-i-1}$. 

The arguments in the proof of Theorem~\ref{s2Thm}~(ii) now apply verbatim to show that,
in both cases we just considered, i.e., when either
$g_1\in
\{T\in\calM^{(2)}\,:\, z_1/z_2<\a_-\}$ or
$g_1\in \{T\in\calM\,:\, \a_-\le z_1/z_2\le \a_+\}$, the limit of $\{g_{-i}\}_{i=-1}^\infty$ exists in the smooth topology and is an Einstein metric. Moreover, similar arguments to those employed in part~(i) above show that actually
\begin{equation}
\begin{aligned}
\label{MinfproofEq}
\{T\in\calM^{(2)}\,:\, z_1/z_2<\a_-\}\cup \{T\in\calM\,:\, \a_-\le z_1/z_2\le \a_+\}
=
\calM^{(\infty)}.
\end{aligned}
\end{equation}
Indeed, $\gamma_{22}^1>0$ by \eqref{gamma211Eq},
and so the argument employed in part~(i) goes through in the same way to show that any $T$ with 
$z_1/z_2> \a_+$ has finite Ricci index, while 
if $z_1/z_2<\a_-$ and $T\not\in\calM^{(2)}$, then,
of course, $r(T)=1<\infty$. 
Finally, Lemma \ref{TrivCaseLemma} allows us to replace
$\calM^{(2)}$ in~\eqref{MinfproofEq} by $\calM$ since it shows that $T\in \calM^{(2)}$ when $z_1/z_2<\a_-$. This proves~\eqref{Minfnonmax}.

Suppose that $\Ad_G(H)|_{\mathfrak m_1}$ is non-trivial
and $\calE$ is empty. Assume $r(g_1)\ge2$. Then, Lemma~\ref{lem_nonmax}~(iii)
shows that $\a_0>\a_1$. Arguing as in the previous paragraphs, we show that
$r(g_1)<\infty$ (since $\gamma_{22}^1>0$ by \eqref{gamma211Eq}).
Thus, $\calM^{(\infty)}=\emptyset$.

Next, 
if $\Ad_G(H)|_{\mathfrak m_1}$ is trivial,
then $\a_-=\a_+>0$ (recall Lemma \ref{lem_nonmax} (ii-a)). 
Once again, since $\gamma_{22}^1>0$ by \eqref{gamma211Eq}, we see that any $T$ with 
$z_1/z_2> \a_+$ has finite Ricci index. 
Suppose now $T$ is such that
$z_1/z_2< \a_+$. Lemma~\ref{TrivCaseLemma} implies that $r(T)\ge 2$. Denote $g_1:=T$ and $g_0:=\Ric T$. Lemma~\ref{lem_nonmax}~(ii-d) shows that $\a_0<\a_1$, and applying
Lemma \ref{TrivCaseLemma} again, we obtain $r(T)\ge 3$.
By induction,~\eqref{Minfminus} must hold.
Also, the ancient Ricci iteration~\eqref{RRIEq} 
exists, and
$\a_{-i}\in (0,\a_-)$ for all $i\in\NN\cup\{0\}$. We have $\a_{-i}>\a_{-i-1}$ and, therefore, $\lim_{i\ra\infty}\a_{-i}=0$. To understand the limit of the sequence $\{g_{-i}\}_{i=-1}^\infty$, we analyze more closely formul\ae~\eqref{PRC_alg_max3}. Because $\Ad_G(H)|_{\mathfrak m_1}$ is trivial, 
$\zeta_1=0=\gamma_{11}^1$. 
By~\eqref{gamma112Eq}, $\gamma_{11}^2=0$.
Lemma \ref{thetaLemma}, together with~\eqref{A1Eq} and \eqref{gamma211Eq} thus imply
\begin{align*}
A_1&=B_2=C_1=0, \\
4d_1d_2A_2&=\th_2=4d_1d_2\zeta_2+d_1\gamma_{22}^2+4d_1\gamma_{22}^1\ge 
4d_1\gamma_{22}^1>0.
\end{align*}
Consequently, $\{g_{-i}\}_{i=-1}^\infty$ converges smoothly to the degenerate tensor $A_2\pi_2^*Q$.
This collapsed limit is the pull-back of a metric $g_E$ on $G/K$ under the inclusion
map $G/K\hookrightarrow G/H$. Since 
$G/K$ is isotropy irreducible (i.e., $s=1$ for it), all 
$G$-invariant metrics on it, and hence also $g_E$,  are Einstein.
Finally,  $(G/H,g_{-i})$ must converge in the Gromov--Hausdorff topology to $(G/K,g_E)$
\cite[Proposition 2.6] {MB14}. This concludes the proof of Theorem~\ref{AncientThm}.
\end{proof}

\subsection{Examples}
\label{examSubSec}

In this subsection, we assume that the group $G$ is simple and that the inner product $Q$ coincides with $-B$, the negative of the Killing form.
Let~\eqref{HKGEq} and~\eqref{km1hEq} hold for some connected Lie subgroup $K<G$ with Lie algebra $\mathfrak k$. We first consider a situation where $\Ad_G(H)|_{\mathfrak m_1}$ is trivial.

Recall that $d_i$ denotes the dimension of the $i$-th summand in the decomposition~\eqref{m_decomp}. The numbers $\gamma_{ij}^k$ are the structure constants of $M$ defined by~\eqref{gamma_def}. The constant $\zeta_i$ is the eigenvalue of the Casimir operator on $\mathfrak m_i$ (see~\eqref{zeta_def}). It is~0 if and only if $\Ad_G(H)|_{\mathfrak m_i}$ is trivial. As Lemma~\ref{RicPosEinLemma} shows, the Einstein metrics in $\mathcal M$ are characterised by the solutions of the equation $P(x,x)=0$, where the polynomial $P(x,y)$ is given by~\eqref{Pxx_def}.

\begin{example}[collapsing] {\rm
\label{example_triv}
Suppose
\begin{align*} G=SO(2m),\qquad K=U(m),\qquad H=SU(m),\qquad m\in\mathbb N\cap[3,\infty).
\end{align*} 
We identify $K$ and $H$ with subgroups of $G$ in the natural way. The reader will find a description of the isotropy representation of $G/H$ in~\cite[Example~I.24]{WDMK08}. In particular, $\Ad_G(H)|_{\mathfrak m_1}$ is trivial, and we have
$$\zeta_1=0,\q d_1=1,\q d_2=m^2-m.$$
Because $\mathfrak m_1$ is 1-dimensional, the constant $\gamma_{11}^1$ equals~0. Formul\ae\ \eqref{Casimir} and~\eqref{gamma112Eq} imply
\begin{align*}
\gamma_{22}^1=d_1=1.
\end{align*}
Since $G/K$ is symmetric, the inclusion $[\mathfrak m_2,\mathfrak m_2]\subset\mathfrak k$ holds, and $\gamma_{22}^2$ vanishes.
The expression $P(x,x)$ (recall~\eqref{Pxx_def} and~\eqref{gamma112Eq}) is now given by the formula
\begin{align*}
P(x,x)&=d_2\gamma_{22}^1x^4+2d_1\gamma_{22}^1x^4+(2d_1d_2-2d_2\gamma_{22}^1-2d_1d_2x)x^2 \\ &=(m^2-m+2)x^4-2(m^2-m)x^3.
\end{align*}
Lemma~\ref{RicPosEinLemma} implies (recall~\eqref{alphapmEq})
$$\a_-=\a_+=\frac{2(m^2-m)}{m^2-m+2}\,.$$ 
Theorem~\ref{s2Thm}~(ii-a) describes the Ricci iteration on~$M$. It shows that, given $T\in\mathcal M$, there exists a unique sequence $\{g_i\}_{i\in\mathbb N}$ satisfying~\eqref{RIEq} for all $i\in\mathbb N$ and $g_1=cT$ for some $c>0$. This sequence converges to an Einstein metric proportional to $g^{\a_-}=g^{\a_+}$ (recall~\eqref{gxEq}). 
Theorem~\ref{AncientThm}~(ii-a) shows that 
$$\mathcal M^{(\infty)}=\Big\{T\in\mathcal M\,:\, z_1/z_2\le\frac{2(m^2-m)}{m^2-m+2}\Big\},$$
where $z_1,z_2$ are from~\eqref{T_def}. 
By formula~\eqref{Casimir},
$$\zeta_2=\frac1{2d_2}(-2\gamma_{22}^1+d_2)=\frac{m^2-m-2}{2(m^2-m)}.$$
When $g_1\in\mathcal M^{(\infty)}$, the ancient Ricci iteration $\{g_{-i}\}_{i=-1}^\infty$ given by~\eqref{RRIEq}
converges smoothly to the degenerate metric
$$\bigg(\zeta_2+\frac{\gamma_{22}^1}{d_2}\bigg)\pi_2^*Q=- \frac12\pi_2^*B.$$
Then,
$
SO(2m)/SU(m)
$ 
collapses to
$SO(2m)/U(m)$. Note that $\dim SO(6)/SU(3)=7$.
} \end{example}

For the next example,
we introduce some additional notation largely following~\cite[\S3]{MWWZ86}. The space $K/H$ may not be effective, and we denote by $K'$ the quotient of $K$ acting effectively on $K/H$. Assume $K'$ is semisimple and $B_{\mathfrak k'}=\alpha B|_{\mathfrak k'}$ for some $\alpha>0$, where $\mathfrak k'$ is the Lie algebra of $K'$ and $B_{\mathfrak k'}$ is the Killing form of $\mathfrak k'$. Consider two Casimir operators
$$
C_{\mathfrak m_1,-B_{\mathfrak k'}|_{\mathfrak h}}:=-
\sum_{j=1}^{q_H}\ad u_j\circ\ad u_j,
\qquad C_{\mathfrak m_2,-B|_{\mathfrak k}}:=-
\sum_{j=1}^{q_K}\ad v_j\circ\ad v_j,
$$
where $\{u_j\}_{j=1}^{q_H}$ is an orthonormal basis of $\mathfrak h$ with respect to $-B_{\mathfrak k'}|_{\mathfrak h}$, and $\{v_j\}_{j=1}^{q_K}$ is an orthonormal basis of $\mathfrak k$ with respect to $-B|_{\mathfrak k}$. These operators have domains $\mathfrak m_1$ and $\mathfrak m_2$, respectively. Let $\zeta_1^*$ and $\zeta_2^*$ be the numbers such that
\begin{align*}
C_{\mathfrak m_1,-B_{\mathfrak k'}|_{\mathfrak h}}
=\zeta_1^*\, \h{id},\qquad C_{\mathfrak m_2,-B|_{\mathfrak k}}
=\zeta_2^*\, \h{id}.
\end{align*}
As shown in~\cite[page~188]{MWWZ86},
\begin{align}\label{newconstantsEq}
\zeta_1&=\alpha\zeta_1^*, \notag\\
\gamma_{22}^1&=d_1(1-\alpha), \notag\\
\gamma_{11}^1&=d_1-\gamma_{22}^1-2d_1\zeta_1=d_1\alpha(1-2\zeta_1^*), \notag\\
\gamma_{22}^2&=d_2-2d_2\zeta_2^*.
\end{align}

\begin{example}[non-collapsing] {\rm
\label{example_nontriv}
Assume
\begin{align*} G&=SO(2m-1),\qquad K=SO(2m-2), \qquad H=U(m-1),\qquad m\in\mathbb N\cap[3,\infty).
\end{align*}
We identify $K$ and $H$ with subgroups of $G$ in the natural way. For descriptions of the isotropy representation of $G/H$, see~\cite[Example~6]{MWWZ86} and~\cite[Example~I.18]{WDMK08}. Let us first find the relevant constants. We have
\begin{align*}
\alpha=\frac{2m-4}{2m-3},\qquad \zeta_1^*=\zeta_2^*=\frac12, \q d_1=(m-1)(m-2),\qquad d_2=2(m-1),
\end{align*}
see~\cite[p.~337]{WDMK08}. In light of~\eqref{newconstantsEq}, this yields
\begin{align*}
\gamma_{22}^1=\frac{(m-1)(m-2)}{2m-3}, \qquad 
\gamma_{11}^1=\gamma_{22}^2=0.
\end{align*}
The expression $P(x,x)$ (recall~\eqref{Pxx_def} and~\eqref{gamma112Eq}) is now given by the formula
\begin{align*}
P(x,x)&=d_2\gamma_{22}^1x^4+2d_1\gamma_{22}^1x^4+(2d_1d_2-2d_2\gamma_{22}^1
-2d_1d_2x)x^2 
\\
&=(d_2+2d_1)\gamma_{22}^1x^4-2d_1d_2x^3+2d_2(d_1-\gamma_{22}^1)x^2 \\ 
&=\frac{2(m-1)^3(m-2)}{2m-3}x^4-4(m-1)^2(m-2)x^3+\frac{8(m-1)^2(m-2)^2}{2m-3}x^2.
\end{align*}
The equality $P(x,x)=0$ holds for some $x>0$ if and only if 
\begin{align}\label{exmpl_eq}
\frac{m-1}{2m-3}x^2-2x+\frac{4(m-2)}{2m-3}=0.
\end{align}
This is a quadratic equation in $x$ with discriminant
\begin{align*}
\mathcal D=
\frac4{(2m-3)^2}.
\end{align*}
Its solutions are the positive numbers
\begin{align*}
\alpha_\pm=\frac{(2m-3)\big(1\pm\sqrt{\mathcal D/4}\big)}{m-1}=\frac{2m-3\pm1}{m-1}.
\end{align*}
These numbers satisfy~\eqref{alphapmEq}. The representation $\Ad_G(H)|_{\mathfrak m_1}$ is nontrivial. Indeed, otherwise, the irreducibility of $\Ad_G(H)|_{\mathfrak m_1}$ would imply that $d_1=1$, which is impossible for any choice of~$m$. Theorem~\ref{s2Thm}~(ii-b) describes the Ricci iteration on~$M$. In particular, it shows that a sequence $\{g_i\}_{i\in\mathbb N}$ satisfying~\eqref{RIEq} for all $i\in\mathbb N$ and $g_1=cT$ for some $c>0$ exists if and only if (recall~\eqref{eq_alpha_etc})
\begin{align*}
\alpha_T\ge\alpha_-=\frac{2m-4}{m-1}\in[1,2).
\end{align*}
Theorem~\ref{AncientThm}~(ii-b) describes the ancient Ricci iterations on~$M$. It demonstrates that 
$$\mathcal M^{(\infty)}=\big\{T\in\mathcal M\,:\, z_1/z_2\le\alpha_+=2\big\},$$
where $z_1,z_2$ are given by~\eqref{T_def}.
Note that $\dim SO(5)/U(2)=6$.
} \end{example}

\section{Relative compactness for maximal isotropy}
\lb{GenSec}

This section contains the proof of Theorem~\ref{GenThm}. We no longer impose an upper bound on the number of summands in the isotropy representation. Nor do we require the pairwise inequivalence of the summands.

\begin{remark}
The arguments in this section actually yield a statement that is stronger than the relative compactness of the sequences $\{g_i\}_{i\in\NN}$ and $\{g_{-i}\}_{i=-1}^\infty$ satisfying the iteration equation~\eqref{RIEq} and the ancient iteration equation~\eqref{RRIEq}. Namely, consider the set
\begin{align}\label{Xi_def}
\Xi=\{h\in\mathcal M\,:\,h=\Ric\tilde h~\mbox{for some}~\tilde h\in\mathcal M~\mbox{and}~S(h)\ge0\},
\end{align}
where $S(h)$ denotes the scalar curvature of~$h$. We show that this set is compact in the topology it inherits from~$\mathcal M$. Of course, if $\{g_i\}_{i\in\NN}$ satisfies~\eqref{RIEq} (or, if $\{g_{-i}\}_{i=-1}^\infty$ satisfies \eqref{RRIEq}), 
then $\{g_i\}_{i=2}^\infty$ (or $\{g_{-i}\}_{i=0}^\infty$) lies in $\Xi$.
\end{remark}

The proof of Theorem~\ref{GenThm} will be based on the following three lemmas.
Denote
\begin{align*}
\Theta:=\{X\in T_\mu M\,:\,Q(X,X)=1\},
\end{align*}
the $Q$-unit sphere in $T_\mu M$.

\begin{lemma}\label{lem_ratio}
Suppose $H$ is maximal in $G$. If $g\in\mathcal M$ has nonnegative scalar curvature, then 
\begin{align*}
\frac{\inf_{X\in\Theta}g(X,X)}{\sup_{X\in\Theta}g(X,X)}\ge a,
\end{align*}
for some constant $a\in(0,1]$ depending only on $G$, $H$ and $Q$.
\end{lemma}

\begin{proof}
Choose the decomposition~\eqref{m_decomp} so that~\eqref{gEq} is satisfied
with $x_1\le\cdots\le x_s$ (see~\cite[p.~180]{MWWZ86}).
Consider the $G$-invariant metric
\begin{align*}
\hat g:=\frac1{x_s}g
\end{align*}
on $M$.
By definition, 
$\inf_{X\in\Theta}\hat g(X,X)\le1$ and $\sup_{X\in\Theta}\hat g(X,X)=1.$
According to~\cite[Lemma~2.4]{AP15}, these formul\ae\ imply the inequality
\begin{align*}
S(\hat g)&\le a_1\Big(\frac{x_1}{x_s}\Big)^{-1}-a_2\bigg(\Big(\frac{x_1}{x_s}\Big)^{-\frac{2^{s-1}}{2^{s-1}-1}}+1\bigg) \\ 
&\le a_1\Big(\frac{x_1}{x_s}\Big)^{-1}-a_2\Big(\frac{x_1}{x_s}\Big)^{-\frac{2^{s-1}}{2^{s-1}-1}}.
\end{align*}
Here, $S(\hat g)$ denotes the scalar curvature of $\hat g$, and $a_1,a_2>0$ are constants depending only on $G$, $H$ and $Q$. As $S(g)\ge0$, also $S(\hat g)\ge0$, and
\begin{align*}
a_2\Big(\frac{x_1}{x_s}\Big)^{-\frac{2^{s-1}}{2^{s-1}-1}}\le a_1\Big(\frac{x_1}{x_s}\Big)^{-1},
\end{align*}
so
\begin{align*}
\frac{\inf_{X\in\Theta}g(X,X)}{\sup_{X\in\Theta}g(X,X)}=\frac{x_1}{x_s}\ge\Big(\frac{a_2}{a_1}\Big)^{2^{s-1}-1},
\end{align*}
as desired.
\end{proof}

\begin{lemma}\label{prop_closed}
Let $H$ be maximal in $G$. Suppose $\{h_i\}_{i\in\NN}\subset\mathcal T$ is a sequence of positive-semidefinite (0,2)-tensor fields on $M$ converging smoothly to some (0,2)-tensor field~$h\in\mathcal T$. If $h_i=\Ric\tilde h_i$ with $\tilde h_i\in\mathcal M$ for every $i\in\mathbb N$, then there exists $\tilde h\in\mathcal M$ satisfying the equality $h=\Ric\tilde h$.
\end{lemma}

The inner product $Q$ induces an inner product on $\mathfrak m^*\otimes\mathfrak m^*$, which yields, via~\eqref{barMEq}, an inner product on~$\mathcal T$. We denote by $|\cdot|_Q$ the corresponding norm on~$\mathcal T$.

\begin{proof}
By renormalizing each $\tilde h_i$ (leaving $h_i=\Ric\tilde h_i$ unchanged), we assume that
$|\tilde h_i|_Q=1$.
Because the space $\mathcal T$ is finite-dimensional, the unit sphere in it with respect to $|\cdot|_Q$ is compact. Consequently, there exists a subsequence $\{\tilde h_{i_m}\}_{m\in\NN}$ converging to some $\tilde h\in\mathcal T$. We claim that $\tilde h\in\calM$, 
or, equivalently, that
\begin{align*}
\inf_{X\in\Theta}\tilde h(X,X)>0.
\end{align*}
By Lemma~\ref{lem_ratio}, 
\begin{align*}
\inf_{X\in\Theta}\tilde h_{i_m}(X,X)\ge a\sup_{X\in\Theta}\tilde h_{i_m}(X,X)\ge a\frac{|\tilde h_{i_m}|_Q}{\sqrt n}=\frac a{\sqrt n}>0,
\end{align*}
and so the same holds for $\tilde h$ by passing to the limit.
Since
\begin{align*}
|\Ric\tilde h-h|_Q=
\lim_{m\to\infty}|\Ric\tilde h_{i_m}-h|_Q=\lim_{m\to\infty}|h_{i_m}-h|_Q=0,
\end{align*}
one has $h=\Ric\tilde h$.
\end{proof}

\begin{lemma}\label{lem_pinch}
Let $H$ be maximal in $G$. Assume $g$ lies in the set $\Xi$ given by~(\ref{Xi_def}).
Then
\begin{align*}
a_-\le\inf_{X\in\Theta}g(X,X)\le\sup_{X\in\Theta}g(X,X)\le a_+
\end{align*}
for some constants $a_-,a_+>0$ depending only on $G$, $H$ and $Q$.
\end{lemma}

\begin{proof}
First, we prove that
\begin{align}\label{pinch1Eq}
\inf_{X\in\Theta}g(X,X)\ge a_-.
\end{align}
Assume, for the sake of  contradiction, that there exists a sequence $\{h_i\}_{i\in\NN}\subset\Xi$ such that 
$$\inf_{X\in\Theta}h_i(X,X)\le\frac1i,\q i\in\mathbb N.$$
According to Lemma~\ref{lem_ratio},
\begin{align*}
\sup_{X\in\Theta}h_i(X,X)\le\frac1{ia},\q i\in\mathbb N.
\end{align*}
Therefore, the formula $\lim_{i\to\infty}|h_i|_Q=0$ holds, and the sequence $\{h_i\}_{i\in\NN}$ converges 
to~$0\in\mathcal T$. In view of Lemma~\ref{prop_closed},  $\Ric\tilde h=0$ for some $\tilde h\in\mathcal M$. However, by Lemma~\ref{RicPosLemma}, since $H$ is maximal in $G$, there are no Ricci-flat metrics in $\mathcal M$. This contradiction proves~\eqref{pinch1Eq}.

Next, we claim that
\begin{align}\label{pinch2_Eq}
\sup_{X\in\Theta}g(X,X)\le a_+.
\end{align}
Consider the set
\begin{align*}
\Omega:=\big\{h\in\mathcal M\,:\,\sup_{X\in\Theta}h(X,X)=1,~S(h)\ge0\big\}.
\end{align*}
Lemma~\ref{lem_ratio} shows that
$\Omega\subset\hat\Omega:=\{h\in\mathcal M\,:\,a\le h(X,X)\le1~\mbox{for all}~X\in\Theta\},
$
where $a>0$. 
The topology of $\mathcal M$ induces topologies on $\Omega$ and $\hat\Omega$. Fix a $Q$-orthonormal basis $\{e_l\}_{l=1}^n$ in $T_\mu M$. Clearly, $h\in\mathcal M$ lies in $\hat\Omega$ if and only if the eigenvalues of the matrix of $h$ at $\mu\in M$ with respect to the basis $\{e_l\}_{l=1}^n$ belong to the interval $[a,1]$. This observation implies that $\hat\Omega$ is compact. It is easy to see that $\Omega$ is closed in $\hat\Omega$. Therefore, $\Omega$ must be compact.

Define a function 
\begin{align*}
f:\Omega\to\mathbb R
\end{align*}
so that $f(h)$ is the largest of the sectional curvatures of $h$ at $\mu\in M$. This function is continuous and, therefore, bounded. Given $h\in\Omega$, it is easy to see that $bh$ has sectional curvatures less than $\frac1{n-1}$ as long as $b>(n-1)f(h)$. 
For such $b$, the tensor field
$bh$ cannot be the Ricci curvature of any metric in 
$\mathcal M$~\cite[Corollary 3.6]{DeTKoiso},~\cite[Theorem 4.3]{RH84}.

Define
\begin{align*}
\tilde g:=g/{\sup_{X\in\Theta}g(X,X)}.
\end{align*}
It is clear that $\tilde g$ lies in $\Omega$. According to the assumptions of the lemma, $b\tilde g$ is the Ricci curvature of a metric in $\mathcal M$ for $b=\sup_{X\in\Theta}g(X,X)$.
Consequently, 
\begin{align*}
\sup_{X\in\Theta}g(X,X)\le(n-1)f(\tilde g).
\end{align*} 
Thus,~\eqref{pinch2_Eq} holds with
$
a_+:=(n-1)\max_{h\in\Omega}f(h).
$
\end{proof}

\begin{proof}[Proof of Theorem~\ref{GenThm}.]
(i) The existence of a sequence $\{g_i\}_{i\in\NN}\subset\mathcal M$ satisfying~\eqref{RIEq} for all $i\in\NN$ and $g_1=cT$ for some $c>0$ is a consequence of Theorem~\ref{max_uniqThm} and the maximality assumption on~$H$. Lemma~\ref{lem_pinch} implies that any such sequence lies in the set
\begin{align*}
\{cT\}\cup\{h\in\mathcal M\,:\,a_-\le h(X,X)\le a_+~\mbox{for all}~X\in\Theta\}.
\end{align*}
This set, with the topology inherited from $\mathcal M$, is compact. Any of its subsets must, therefore, be relatively compact in~$\mathcal M$.

(ii) The proof is completely analogous to that of~(i).
\end{proof}

The following remark may prove useful in the future analysis of the Ricci iteration. Given $i\in\mathbb Z$, the Riemannian metric $g_i$ from Theorem~\ref{GenThm} induces an inner product on $\mathfrak m^*\otimes\mathfrak m^*$, which yields, via~\eqref{barMEq}, an inner product on $\mathcal T$. Denote the corresponding norm on $\mathcal T$ by~$|\cdot|_{g_i}$.

\begin{remark}
Lemma~\ref{lem_pinch} provides an estimate for the difference between $g_i$ and $g_{i-1}$. Namely, we have
\begin{align}\label{ineq_rem}
|g_i-g_{i-1}|_Q^2\le a_+^2\big(n+|\Ric g_i|_{g_i}^2-2S(g_i)\big),\qquad i\in\mathbb Z.
\end{align}
To see this, fix a $Q$-orthonormal basis $\{e_l\}_{l=1}^n$ of $T_\mu M$ diagonalizing $g_i$. The left-hand side of~\eqref{ineq_rem} equals
\begin{align*}
\sum_{l,m=1}^n(g_i-g_{i-1})(e_l,e_m)^2
=\sum_{l,m=1}^ng_i(e_l,e_l)g_i(e_m,e_m)\bigg((g_i-g_{i-1})\bigg(\frac{e_l}{\sqrt{g_i(e_l,e_l)}},\frac{e_m}{\sqrt{g_i(e_m,e_m)}}\bigg)\bigg)^2.
\end{align*}
Lemma~\ref{lem_pinch} implies that it is bounded from above by
\begin{align*}
a_+^2\sum_{l,m=1}^n(g_i-g_{i-1})\bigg(\frac{e_l}{\sqrt{g_i(e_l,e_l)}},\frac{e_m}{\sqrt{g_i(e_m,e_m)}}\bigg)^2.
\end{align*}
The vectors $\Big\{\frac{e_l}{\sqrt{g_i(e_l,e_l)}}\Big\}_{l=1}^n$ form a $g_i$-orthonormal basis of $T_\mu M$. Therefore, our last displayed expression equals
\begin{align*}
a_+^2|g_i-g_{i-1}|_{g_i}^2&= a_+^2\langle g_i-g_{i-1},g_i-g_{i-1}\rangle_{g_i}\\
&=a_+^2\big(|g_i|_{g_i}^2+|g_{i-1}|_{g_i}^2-2\langle g_i,g_{i-1}\rangle_{g_i}\big)
=a_+^2\big(n+|\Ric g_i|_{g_i}^2-2S(g_i)\big),
\end{align*}
where $\langle\cdot,\cdot\rangle_{g_i}$ is the inner product on $\mathcal T$ induced by~$g_i$. Formula~\eqref{ineq_rem} follows immediately.
\end{remark}

\section*{Acknowledgments}
\label{}

We are grateful to the referees for their careful reading and comments.
Research supported by the Australian Research Council Discovery Early-Career Researcher Award DE150101548 (A.P.) and by NSF grants DMS-1206284,1515703 and a Sloan Research Fellowship (Y.A.R.). 
Part of this work took place while Y.A.R. visited MSRI (supported by NSF grant DMS-1440140)
during the Spring 2016 semester.

\vspace{10pt}

{\sc
The University of Queensland
}

\texttt{a.pulemotov@uq.edu.au}

\medskip

{\sc 
University of Maryland
}

{\texttt{yanir@umd.edu}}

\end{document}